\documentclass[leqno]{article}

\usepackage{CJKutf8}
\usepackage{a4wide}
\usepackage[]{amsfonts} \usepackage[]{amsmath}
\usepackage[]{amssymb} \usepackage[]{latexsym}
\usepackage{graphicx,color} \usepackage{amsthm}
\usepackage{url}
\usepackage{graphicx}
\usepackage[numbers]{natbib}
\usepackage{bm}
\usepackage{amssymb,amsmath,color,times}
 \usepackage{mathptmx}      
\usepackage{latexsym}
  \usepackage[title]{appendix}

\begin{document}
\begin{CJK}{UTF8}{gbsn}

\theoremstyle{plain}
\newtheorem{theorem}{Theorem}[section] \newtheorem*{theorem*}{Theorem}
\newtheorem{proposition}[theorem]{Proposition} \newtheorem*{proposition*}{Proposition}
\newtheorem{lemma}[theorem]{Lemma} \newtheorem*{lemma*}{Lemma}
\newtheorem{corollary}[theorem]{Corollary} \newtheorem*{corollary*}{Corollary}

\theoremstyle{definition}
\newtheorem{definition}[theorem]{Definition} \newtheorem*{definition*}{Definition}
\newtheorem{example}[theorem]{Example} \newtheorem*{example*}{Example}
\newtheorem{remark}[theorem]{Remark} \newtheorem*{remark*}{Remark}
\newtheorem{hypotheses}[theorem]{Hypotheses} \newtheorem{assumption}[theorem]{Assumption}
\newtheorem{notation}[theorem]{Notation} \newtheorem*{question}{Question}

\newcommand{\ds}{\displaystyle} \newcommand{\nl}{\newline}
\newcommand{\eps}{\varepsilon}
\newcommand{\bE}{\mathbb{E}}
\newcommand{\cB}{\mathcal{B}}
\newcommand{\cF}{\mathcal{F}}
\newcommand{\cA}{\mathcal{A}}
\newcommand{\cM}{\mathcal{M}}
\newcommand{\cD}{\mathcal{D}}
\newcommand{\cH}{\mathcal{H}}
\newcommand{\cN}{\mathcal{N}}
\newcommand{\cL}{\mathcal{L}}
\newcommand{\cLN}{\mathcal{LN}}
\newcommand{\bP}{\mathbb{P}}
\newcommand{\bQ}{\mathbb{Q}}
\newcommand{\bN}{\mathbb{N}}
\newcommand{\bR}{\mathbb{R}}
\newcommand{\barsigma}{\overline{\sigma}}
\newcommand{\VIX}{\mbox{VIX}}
\newcommand{\erf}{\mbox{erf}}
\newcommand{\LMMR}{\mbox{LMMR}}
\newcommand{\cLcir}{\mathcal{L}_{\!_{C\!I\!R}}}
\newcommand{\rhor}{\raisebox{1.5pt}{$\rho$}}
\newcommand{\varphir}{\raisebox{1.5pt}{$\varphi$}}
\newcommand{\taur}{\raisebox{1pt}{$\tau$}}
\newcommand{\spx}{S\&P 500 }

\title{Quasi-Monte Carlo-Based Conditional Malliavin Method for Continuous-Time Asian Option Greeks}

\author{Chao Yu \and Xiaoqun Wang}

\maketitle

\begin{abstract}
Although many methods for computing the Greeks of discrete-time Asian options are proposed, few methods to calculate the Greeks of continuous-time Asian options are known. In this paper, we develop an integration by parts formula in the multi-dimensional Malliavin calculus, and apply it to obtain the Greeks formulae for continuous-time Asian options in the multi-asset situation. We combine the Malliavin method with the quasi-Monte Carlo method to calculate the Greeks in simulation. We discuss the asymptotic convergence of simulation estimates for the continuous-time Asian option Greeks obtained by Malliavin derivatives. We propose to use the conditional quasi-Monte Carlo method to smooth Malliavin Greeks, and show that the calculation of conditional expectations analytically is viable for many types of Asian options. We prove that the new estimates for Greeks have good smoothness. For binary Asian options, Asian call options and up-and-out Asian call options, for instance, our estimates are infinitely times differentiable. We take the gradient principal component analysis method as a dimension reduction technique in simulation. Numerical experiments demonstrate the large efficiency improvement of the proposed method, especially for Asian options with discontinuous payoff functions.
\end{abstract}

\section{Introduction}

\label{intro}
In mathematical finance,  there is a growing emphasis on efficient techniques for the computations of price sensitivities, i.e., the Greeks, rather than the price itself. The Greeks are the derivatives of the option price with respect to model parameters, which can be expressed as \begin{align*}
Greek:=\frac{\partial}{\partial \alpha} {\mathbb E}_{\mathbb Q}\left[e^{-rT}H\right],
\end{align*}
where ${\mathbb E}_{\mathbb Q}\left[e^{-rT}H\right]$ is the price of some option with the payoff $H$, and $\alpha$ is some model parameter. The Greeks play an important role in hedging financial derivatives and measuring and managing risk (see \cite{Hull22}). As pointed out by \citet{Glasserman04}, whereas the prices themselves can often be observed in the market, their sensitivities (Greeks) cannot. They could be only obtained by mathematical calculation or statistical methods. Thus, accurate and fast computation of sensitivities is arguably even more important from both theoretical and practical point of view. 

There are numerous ways to calculate the Greeks of discrete path-dependent options,  like the finite difference (FD) method \cite{Glasserman04}, the pathwise (PW) method \cite{Ho83, Broadie96} and the likelihood ratio (LR) method \cite{Glynn87, Broadie96}. However, when handling the Greeks of continuous path-dependent options, some of these methods have limitations in many cases due to the complexity of the continuous-time performance of the risky asset $S_t$. For example, the LR method could not be used directly to calculate the Greeks of continuous-time simple Asian options with payoff  $H=f(\bar S_T)$, where $\bar S_T=\frac{1}{T}\int_0^TS_t\mathrm{d}t$ is the continuous-time average of $S_t$ and $f$ is the payoff function, due to the unknown density of $\bar S_T$ (hereafter we omit the `continuous-time'). Moreover, the FD method produces biased estimates with large variance, and the PW method is invalid when the payoff function $f$ is discontinuous (see \cite{Glasserman04}). 

\citet{Fournie99} developed a new method to calculate the Greeks of options based on Malliavin calculus, and \citet{Benhamou00} generalized the method for the Asian options. The main idea is to use the integration by parts formula in Malliavin calculus to eliminate the need of taking the derivative of the payoff function. By this mean, the Greeks of  options  have the following expression
\begin{eqnarray*}
Greek=\mathbb{E}_{\mathbb Q}\left[e^{-rT}H\cdot weight\right],
\end{eqnarray*}
where $weight$ is some random variable.

\citet{Fournie99} gave a particular solution for the weighting function in the case of the $delta$ of simple Asian options with payoff  $H=f(\bar S_T)$ in the one-asset case. However, there are very few results in the literature with regard to complex Asian options with payoff $H=f(S_T,\bar S_T)$ and those in the multi-asset case. Notice that neither the extension from simple Asian options to complex Asian options nor that from the one-asset case to the multi-asset case is trivial since the original integration by parts formula cannot be applied directly (see Remarks \ref{rmrm1} and \ref{rmrm2} in Section \ref{MVcalculus}).

To our knowledge, only \citet{Benhamou00} considered the $delta$ of complex Asian options. However, the calculation of the weight function is very complicated. For the multi-asset case, although \citet{Xu14} gave the $delta$ and $gamma$ of simple Asian options, the calculation is complex. For complex Asian options in the multi-asset case, there is no result in the literature.

On the other hand, the Malliavin method gives us a chance to derive unbiased estimates of the Asian option Greeks. Nevertheless, we still need to use simulation to calculate Greeks since these estimates have no closed forms. The Monte Carlo (MC) and the quasi-Monte Carlo (QMC) methods are the usual simulation methods to estimate the integrals. Differently from the MC method, the QMC method uses deterministic low discrepancy points instead of random points to estimate the integrals. The QMC method provides a faster asymptotic convergence than MC for many financial problems  (see \cite{Glasserman04}). But the high dimensionality and the lack of smoothness limit the advantage of QMC over MC. The gradient principal component analysis (GPCA) method \cite{Xiao19} and the conditional quasi-Monte Carlo (CQMC) method \cite{Xiao18, Bayer18} are the two main powerful methods to overcome the high dimensionality and the lack of smoothness, respectively.

When applying the MC and the QMC methods to the Asian option Greeks, we need to discretize the path, which is different from the discrete path-dependent options. It leads to two important issues,  namely, how to deal with the bias of the discrete simulation estimates and how to reduce the variance.

For the first issue, the discrete simulation estimates are directly used to approximate the original estimates of Malliavin Greeks in most relevant literature (see \cite{Benhamou00,Xu14}). There is few literature discussing the bias of the simulation estimates. For the second issue, \citet{Fournie99} were the first to develop a variance reduction technique, which localizes the integration by part formula around the singularity, to estimate the Malliavin Greeks in the MC setting. \citet{Xu14} showed the superiority of the QMC method over the MC method. However, as they mentioned, no variance reduction technique is taken. We believe that a more powerful method combining the QMC method with a more general variance reduction technique could be developed.

Based on the multi-dimensional Malliavin calculus, we develop a general integration by parts formula, and apply it to obtain the Greeks formulae for both simple and complex Asian options in the multi-asset Black-Scholes model. In simulation, we prove the asymptotic unbiasedness of the simulation estimates when the payoff function $f$ is continuous with linear growth. Then based on the CQMC method, we develop a new method to calculate the the Asian option Greeks. We smooth the involved payoff functions by taking conditional expectations, which is proved to be continuous under some uniform convergence conditions, and derive new estimates of the Greeks. The main contributions of this paper are as follows:

\begin{itemize}
 \item In $m$-dimensional Malliavin calculus, we develop the integration by parts formula. Then we apply it to calculate the Greeks ($delta$, $gamma$ and $vega$) for both simple and complex Asian options in the $m$-asset Black-Scholes model. In particular, we obtain the Greeks formulae for European options.
\end{itemize}

\begin{itemize}
 \item We convert the estimates of the Malliavin Greeks to the discrete simulation estimates. For payoff functions with linear growth, we prove that the simulation estimates of Asian option Greeks are asymptotic unbiased.
\end{itemize}

\begin{itemize}
 \item We propose to apply the CQMC method to the simulation estimates for Greeks (obtained by Malliavin method) in continuous situation, which is a first attempt to combine these two methods to reduce the variance of estimates. We smooth the integrands for simulation estimates for Greeks by taking conditional expectations. The simulation estimates for Greeks are usually discontinuous. However, after taking conditional expectations, the estimates for many Asian option Greeks have good smoothness. After then some potent dimension reduction techniques (such as the GPCA method) can be combined with the QMC method to obtain good effectiveness, which could save much time in the practical computation.
  \end{itemize}

\begin{itemize}
 \item We show how the conditional expectations can be calculated analytically. Since the simulation estimates for many Asian option Greeks (obtained by Malliavin calculus) satisfy the separation of variables condition, our method has wide practicability.
\end{itemize}

This paper is organized as follows. In Section~\ref{MVcalculus}, we introduce the basic theory of the Malliavin calculus and generalize the integration by parts formula. In Section~\ref{Sec:model}, we use our formula to obtain the Greeks formulae for Asian options in the Black-Scholes model. We introduce the QMC method in Section~\ref{Sec-QMC}. The main idea of our method is introduced in Section~\ref{Sec:QMC-CMV}. The implementation is demonstrated on three examples in Section~\ref{Sub_ill}. Numerical experiments are performed in Section~\ref{sec:NE}. Finally, we summarize our conclusions in Section~\ref{sec:conclusion}.

\section{Malliavin calculus: basic theory and results}
\label{MVcalculus}
In this section, we briefly introduce the basic theory of Malliavin Calculus in the multi-dimensional case. In addition, we develop an integration by parts formula in order to calculate the Greeks for the complex Asian options in Section \ref{Sec:model}. For more details of this subject, we refer to \cite{Nualart06, Huang00, Ocone91}.

For convenience, we consider the classical Wiener space\footnote{In fact, this space is an irreducible Gaussian space $(\Omega,{\cal F},\mathbb{P};H)$. Here, $(\Omega,{\cal F},\mathbb{P})$ is a complete probability space, $H$ is a real separable Hilbert space, $\{U_h\}_{h\in H}$ is a family of Gaussian random variables satisfying $\mathbb{E}[U_h]=0$ and $\mathbb{E}[U_hU_g]=(h,g)_H$ (the inner product of $h$ and $g$ in $H$) for all $h,g\in H$, and ${\cal F}$ is the completion of the $\sigma$-algebra generated by $\{U_h\}_{h\in H}$ with respect to $\mathbb{P}$. We refer to \cite{Huang00} for more details.}$(\Omega, {\cal F},{\mathbb P}; H)$, which induces a filtered probability $(\Omega, {\cal F}, \{{\cal F}_t\}_{0\le t\le T},{\mathbb P})$ and an $m$-dimensional Brownian motion ${\bm W}=\{{\bm W}_t=(W_t^{(1)},\cdots, W_t^{(m)})'; 0\le t\le T\}$. Here, $\Omega$ is the space of all continuous ${\mathbb R}^m$-valued functions ${\bm\omega}$ on $[0,T]$ such that ${\bm\omega}(0)={\bm 0}$, ${\cal F}$ is the completion of the Borel $\sigma$-algebra ${\cal B}(\Omega)$ with respect to the Wiener measure ${\mathbb P}$, ${\bm W}_t({\bm\omega})={\bm\omega}(t)$ is the coordinate process, $\{{\cal F}_t\}_{0\leq t\leq T}$ is the $\mathbb P$-augmentation of the filtration generated by $\bm W$, which satisfies the usual condition (see \cite{Karatzas91}), and $H=L^2([0,T];{\mathbb R}^m)$. 

We denote by $C^\infty_p({\mathbb R }^n)$ the set of all infinitely continuously differentiable functions $\phi$ such that $\phi$ and all of its partial derivatives have
polynomial growth. Let $\cal S$ denote the class of smooth random variables such that a random variable $X\in{\cal S}$ has the form
\begin{equation*}
X=\varphi\left(\int_0^T{\bm h}_1(s)\cdot\mathrm{d}{\bm W}_{s},\cdots, \int_0^T{\bm h}_n(s)\cdot\mathrm{d}{\bm W}_{s}\right),
\end{equation*}
where $n\in {\mathbb N}_+$, $\varphi\in C^\infty_p({\mathbb R }^{m\times n})$ and ${\bm h}_j\in H$ for $j=1,\cdots, n$. Note that ${\cal S}$ is dense in $L^2(\Omega)$. The Malliavin gradient ${\bm D}_tX$ of the smooth random variable $X$ is defined as the $H$-valued random variable ${\bm D}_tX=(D^{(1)}_tX,\cdots,D^{(m)}_tX)'$ with components
\begin{equation*}
D_t^{(i)}X=\sum_{j=1}^n \frac{\partial \varphi}{\partial x_{ij}}\left(\int_0^T{\bm h}_1(s)\cdot\mathrm{d}{\bm W}_{s},\cdots, \int_0^T{\bm h}_n(s)\cdot\mathrm{d}{\bm W}_{s}\right){ h}^{(i)}_j(t),\quad i=1,\cdots,m.
\end{equation*}

We can check that ${\bm D}_t$ is a closable operator from ${\cal S}\subset L^2(\Omega)$ to $L^2(\Omega;H)$. Denote by  $D^{1,2}(\Omega)$ the closure of the class of smooth random variables ${\cal S}$ with respect to the graph norm (see \cite{Brezis10})
\begin{equation*}
\begin{aligned}
\| X\|_{D^{1,2}(\Omega)}:&=\left[\|X\|_{L^2(\Omega)}^2+\|{\bm D}_tX\|^2_{L^2(\Omega;H)} \right]^{\frac{1}{2}}\\
&=\left[ {\mathbb E} \left[|X|^2  \right]+    {\mathbb E}\left[\int_0^T\sum_{i=1}^m|D^{(i)}_tX|^2\mathrm{d}t\right]   \right]^{\frac{1}{2}}.
\end{aligned}
\end{equation*}
Then ${\bm D}_t$ is a closed dense operator from $L^2(\Omega)$ to $L^2(\Omega;H)$ with dense domain $\text{Dom}\ {\bm D}_t=D^{1,2}(\Omega)$, which is a Banach space under the norm $\Vert\cdot\Vert_{D^{1,2}(\Omega)}$.

Similar to the ordinary gradient operator ${\bm\nabla}_n$ in ${\mathbb R}^n$,  there is also a chain rule for the Malliavin gradient ${\bm D}_t$, which is important in the proof of the following integration by parts formula.

\begin{proposition}\label{chainrule}
\textnormal{(Proposition 1.2.4,  \cite{Nualart06})} Let $\varphi: {\mathbb R}^n\rightarrow {\mathbb R}$ be a Lipschitz function. Suppose that ${\bm X}=(X_1,\cdots,X_n)$ is a stochastic vector whose components belong to $D^{1,2}(\Omega)$, and the law of ${\bm X}$ is absolutely continuous with respect to the Borel-Lebesgue measure on ${\mathbb R}^n$. Then $\varphi({\bm X})\in D^{1,2}(\Omega)$, and 
\begin{equation}
{\bm D}_t\varphi({\bm X})=\sum_{j=1}^n\frac{\partial\varphi}{\partial x_j}({\bm X}){\bm D}_tX_j.
\end{equation}
\end{proposition}
 
 Now we define by ${\delta }: L^2(\Omega;H)  \rightarrow L^2(\Omega)$ with domain $\text{Dom}\ { \delta}$ the adjoint of the closed dense operator ${\bm D}_t:L^2(\Omega)\rightarrow L^2(\Omega;H)$ that is called the Malliavin divergence operator. Although $\delta$ is abstract to understand, it coincides with the It\^o integral with respect to ${\bm W}$ (see \cite{Karatzas91}) when ${\bm v}\in L^2(\Omega;H)\cong L^2(\Omega\times [0,T];{\mathbb R}^m)$ is an ${\cal F}_t$-adapted ${\mathbb R}^m$-valued process, namely, ${\delta}({\bm v})=\int_0^T {\bm v}_t\cdot \mathrm{d}{\bm W}_t $. We will keep use of the notation $\int_0^T{\bm v}_t\cdot\mathrm{d}{\bm W}_t:=\delta ({\bm v})$ when ${\bm v}\in \text{Dom}\ \delta $. Thus we also call $\delta $ the Skorohod integral.

The following lemma is the multiplication formula of $\delta$, which can be very useful to the calculations of Malliavin Greeks in Section \ref{Sec:model}.  

\begin{lemma}
\label{multidelta}
\textnormal{(Proposition 1.3.3,  \cite{Nualart06})} Let $X\in D^{1,2}(\Omega)$ and ${\bm v}\in \text{Dom}\ {\delta}$. Suppose that $X{\bm v}\in L^2(\Omega;H)$ and $X\delta({\bm v})- ( {\bm D} X ,  {\bm v})_H$ is square integrable. Then $X{\bm v}\in \text{Dom}\ {\delta}$ and 
\begin{equation}
{\delta}(X{\bm v})=X\delta({\bm v})- ( {\bm D} X ,  {\bm v})_H.
\end{equation}
\end{lemma}

In Proposition 6.2.1 of \cite{Nualart06}, Nualart gave an integration by parts formula, which plays an important role in calculating the Greeks of options in mathematical finance. Here we extend it to the high dimensional case in the following theorem, which will be strongly used in Section \ref{Sec:model} especially for the case of the complex Asian options in the multi-asset case.
 
\begin{theorem}\label{L0}
Let ${\bm X}_1,\cdots,{\bm X}_n$ be $m$-dimensional random row vectors whose components belong to $D^{1,2}(\Omega)$, and ${\bm Y}$ be an $n$-dimensional random vector. Consider an $H_1$-valued random vector ${\bm v}_t=(v^{(1)}_t,\cdots,v^{(n)}_t)$, where $H_1=L^2([0,T])$. Assume that the following conditions are satisfied
\begin{itemize}
 \item[(a)] There exists an invertible $m\times m$ real matrix ${\bm c}=({\bm c}_1,\cdots,{\bm c}_m)$ with inverse ${\bm b}=({\bm b}_1,\cdots,{\bm b}_m)'$, such that for all $j=1,\cdots,n$,  ${\bm D}_t{\bm X}_j=({\bm c}_1 X_t^{(j1)},\cdots,{\bm c}_mX_t^{(jm)})$ for some   $X^{(jk)}_t\in L^2(\Omega;H_1)$, $k=1,\cdots,m$;
\end{itemize}
\begin{itemize}
 \item[(b)] The law of $({\bm X}_1,\cdots,{\bm X}_n)$ is absolutely continuous with respect to the Borel-Lebesgue measure on ${\mathbb R}^{mn}$;
\end{itemize}
\begin{itemize}
 \item[(c)]  For all $k=1,\cdots,m$, the $n\times n$ matrix ${\bm\zeta}^{(k)}=(\zeta_{ij}^{(k)})_{n\times n}:=((X_t^{(jk)},v^{(i)}_t)_{H_1})_{n\times n}$ is invertible a.s. with inverse $\tilde{ {\bm\zeta}}^{(k)}$, and $\sum_{k=1}^n\sum_{j=1}^{n}Y_{j}\tilde\zeta_{jk}^{(l)}v_t^{(k)}{\bm b}_l' \in \text{Dom}\ \delta$, where $(\cdot,\cdot)_{H_1}$ is the inner product in $H_1$. 
 \end{itemize}
Then for any Lipschitz function $\varphi:{\mathbb R}^{mn}\rightarrow {\mathbb R}$, we have for all $ l=1,\cdots, m$,
\begin{equation}\label{byparts}
\mathbb{E}\left[  \sum_{j=1}^n\frac{\partial\varphi}{\partial x_{jl}} ({\bm X}_1,\cdots,{\bm X}_{n})Y_j \right]=\mathbb{E}\left[  \varphi({\bm X}_1,\cdots,{\bm X}_{n})\delta \left(\sum_{k=1}^n\sum_{j=1}^{n}Y_{j}\tilde\zeta_{jk}^{(l)}v_t^{(k)}{\bm b}_l' \right) \right].
\end{equation}
\end{theorem}
 \begin{proof}
By Proposition \ref{chainrule}, we have 
\begin{equation*} 
\begin{aligned} 
&{\bm D}_t\varphi({\bm X}_1,\cdots,{\bm X}_n)\\
=&
\left(
      \begin{array}{ccc}
       {\bm D}_t{\bm X}_1 ,&\cdots, &{\bm D}_t{\bm X}_n
      \end{array}
      \right)    
      \left(
      \begin{array}{c}
        \frac{\partial\varphi}{\partial{\bm x}_1}({\bm X}_1,\cdots,{\bm X}_n)'\\
       \vdots\\
       \frac{\partial\varphi}{\partial{\bm x}_n}({\bm X}_1,\cdots,{\bm X}_n)'
      \end{array}
      \right)    \\
      =&\left(
      \begin{array}{ccc}
     ({\bm c}_1X^{(11)}_t ,\cdots, {\bm c}_mX_t^{(1m)}),&\cdots, &  ({\bm c}_1X^{(n1)}_t ,\cdots, {\bm c}_mX_t^{(nm)})
      \end{array}
      \right)    
      \left(
      \begin{array}{c}
        \frac{\partial\varphi}{\partial{\bm x}_1}({\bm X}_1,\cdots,{\bm X}_n)'\\
       \vdots\\
       \frac{\partial\varphi}{\partial{\bm x}_n}({\bm X}_1,\cdots,{\bm X}_n)'
      \end{array}
      \right)  .
\end{aligned}
\end{equation*} 
Multiplying both sides of the above equation by ${\bm b}_l$ yields
\begin{equation}
\begin{aligned}
\sum_{i=1}^mb_{li}D^{(i)}_t\varphi({\bm X}_1,\cdots,{\bm X}_n)=\sum_{j=1}^n\frac{\partial\varphi}{\partial x_{jl}}({\bm X}_1,\cdots,{\bm X}_n)X_t^{(jl)} . \label{sum111}
\end{aligned}
\end{equation}
 Take inner product of (\ref{sum111}) and $v^{(k)}_t$ in $L^2([0,T])$ for $k=1,\cdots, n$. Then we have
\begin{equation}\label{sum212}
\begin{aligned}
\left(
      \begin{array}{c}
       \int_0^T\sum_{i=1}^mD^{(i)}_t\varphi({\bm X}_1,\cdots,{\bm X}_n)b_{li}v_t^{(1)}\mathrm{d}t \\
        \vdots   \\
        \int_0^T\sum_{i=1}^mD^{(i)}_t\varphi({\bm X}_1,\cdots,{\bm X}_n)b_{li}v_t^{(n)}\mathrm{d}t 
      \end{array}
      \right)         
     =
     \left(
      \begin{array}{ccc}
       \zeta_{11}^{(l)}& \cdots  &\zeta_{1n}^{(l)} \\
        \vdots  &   \ddots  & \vdots  \\
       \zeta_{n1}^{(l)} & \cdots  & \zeta_{nn}^{(l)}
      \end{array}
      \right)  
         \left(
      \begin{array}{c}
        \frac{\partial\varphi}{\partial{ x}_{1l}}({\bm X}_1,\cdots,{\bm X}_n)\\
       \vdots\\
       \frac{\partial\varphi}{\partial{x}_{nl}}({\bm X}_1,\cdots,{\bm X}_n)
      \end{array}
      \right)  .
\end{aligned}
\end{equation}
Since ${\bm\zeta}^{(l)}$ is invertible a.s. with inverse $\tilde{\bm \zeta}^{(l)}$, we left multiply both sides of (\ref{sum212}) by ${\bm Y}\tilde{\bm \zeta}^{(l)}$. Then formula (\ref{byparts}) follows from the fact that $\delta$ is the adjoint of ${\bm D}_t$.
\end{proof}

\begin{remark}\label{rmrm1}
The Malliavin gradient of ${\bm X}_j$ can be decomposed into a coefficient matrix ${\bm c}$ and a series of processes $(X_t^{(j1)},\cdots, X_t^{jm})$ for $j=1,\cdots,n$ in Theorem \ref{L0}, which is the main extension from the original integration by parts formula in \cite{Nualart06}. That will be more convenient to calculate the Greeks in the multi-asset case in Section \ref{Sec:model}.
\end{remark}

\begin{remark}\label{rmrm2}
It might be unrealistic to obtain a useful integration by parts formula in the high dimensional case without the condition that ${\bm D}_t{\bm X}_j=({\bm c}_1X_t^{(j1)},\cdots,{\bm c}_mX_t^{(jm)})$ for $j=1,\cdots, n$, which means that ${\bm D}_t{\bm X}_1,\cdots,{\bm D}_t{\bm X}_n$ share the same coefficient matrix $\bm c$. This condition allows Theorem \ref{L0} to be used to calculate the complex Asian option Greeks in Section \ref{greekcomplex}.
\end{remark}

If we take $m=n=1$, $c=b=1$ and $X_t=D_tX$ in Theorem \ref{L0}, we can obtain the classical integration by parts formula in the following corollary. 

\begin{corollary}\label{cornu}
\textnormal{(Proposition 6.2.1,  \cite{Nualart06})} Let $m=1$. Suppose that $X\in D^{1,2}(\Omega)$, and ${Y}$ be a random variable. Consider an $H$-valued random variable ${ v}_t$. Assume that $(DX,v)_{H}\neq 0$, $Yv(DX,v)_{H}\in \text{Dom}\ \delta$, and the law of $X$ is absolutely continuous with respect to the Borel-Lebesgue measure on ${\mathbb R}$. Then for any Lipschitz function $\varphi:{\mathbb R} \rightarrow {\mathbb R}$, we have
\begin{equation} 
\mathbb{E}\left[ \varphi'(X)Y\right]=\mathbb{E}\left[  \varphi(X)\delta \left(Yv (DX,v)^{-1}_{H}  \right) \right].
\end{equation}
\end{corollary}

\begin{remark}
When we consider the irreducible Gaussian space $(\Omega,{\cal F},\mathbb{P};H)$, the gradient operator and the divergence operator can be defined similarly. Proposition \ref{chainrule}, Lemma \ref{multidelta} and Corollary \ref{cornu} can also be extended to this case (see \cite{Nualart06}). However, we need an extra condition that $H= H_1\otimes \mathbb{R}^m$ for the extension of Theorem \ref{L0}, where $H_1$ is a real separable Hilbert space. This remark implies that these results can be applied to the classical Gaussian white noise space, which is also an irreducible Gaussian space (see \cite{Huang00,Di Nunno09}).
 \end{remark}

\section{Malliavin Greeks for Asian options}
 \label{Sec:model}
We assume that the financial market consists of one risk-free asset (bond) $B$ and $m$ risky assets (stocks) $S^{(1)}, \cdots, S^{(m)}$. Recall the Black-Scholes model in the multi-asset case (see \cite{Glasserman04,Karatzas91}). Fix a future time $T $ in years. Let ${\bm W}=\{{\bm W_t}=(W_t^{(1)},\cdots,W^{(m)}_t)';0\le t\le T\}$ be an $m$-dimensional Brownian motion on a filtered probability space $(\Omega,{\cal F},\{{\cal F}_t\}_{0\leq t\leq T},\mathbb{P})$, where $\{{\cal F}_t\}$ is the $\mathbb P$-augmentation of the filtration generated by ${\bm W}$, which satisfies the usual condition. The risky assets are the (unique) strong solutions of the following stochastic differential equation with deterministic  initial value for $i=1,\cdots,m$
 \begin{eqnarray*}\label{BS}
 \mathrm{d} S^{(i)}_t=\mu_i S^{(i)}_t  \mathrm{d} t +S^{(i)}_t\sum_{j=1}^m\sigma_{ij}  \mathrm{d}   W^{(j)}_t,\quad S_0^{(i)}>0,\quad  0\le t\le T.　
\end{eqnarray*}
Here, ${\bm\mu}=(\mu_1,\cdots,\mu_m)'$ is a real vector representing the expected return of the risky assets in a short period of time, ${\bm\sigma}=(\sigma_{ij})_{m\times m}$ is a real $m\times m$ matrix representing the volatility of the risky assets. By It\^{o}'s formula (see \cite{Karatzas91}), we can obtain the specific expression of $S_t^{(i)}$ as follows
\begin{equation*}
S_t^{(i)}=S_0^{(i)} \exp\left\{ \left(\mu_i-\frac{1}{2}\sum_{j=1}^m\sigma_{ij}^2\right) t+\sum_{j=1}^m\sigma_{ij} W_t^{(j)} \right\},\quad 0\leq t\leq T,
\label{Stock}　
\end{equation*}
for $i=1,\cdots, m$. The risk-free asset $B$ has the form $B_t=e^{rt}$, where $r>0$ is the risk-free interest rate. We set ${\bm a}:={\bm \sigma}{\bm \sigma}'$ and assume that ${\bm a}$ is positive definite, i.e., 
\begin{equation}
\label{zhengding}
{\bm x}'{\bm a}{\bm x}> 0,\quad \forall {\bm 0}\neq {\bm x}\in {\mathbb R}^m.
\end{equation}

Condition (\ref{zhengding}) implies that ${\bm \sigma}$ is invertible and $M_t:=\exp\{-{\bm \gamma}\cdot {\bm W}_t-\frac{1}{2}\|{\bm\gamma }\|^2t\}$ is a martingale, where ${\bm\gamma}={\bm\sigma}^{-1}({\bm\mu}-r\underset{\tilde{}}{\bm 1})$ and $\underset{\tilde{}}{\bm 1}=(1,\cdots,1)'\in {\mathbb R}^m$, thanks to the Novikov's criterion (see \cite{Karatzas91}). By the Girsanov's theorem (see \cite{Karatzas91}), $\tilde{\bm W}_t:={\bm W}_t+{\bm\gamma}t$ is an $m$-dimensional Brownian motion under the new probability measure (the risk-neutral probability measure) $\mathbb{Q}:=\int M_T\mathrm{d}{\mathbb P}$. Thus for $i=1,\cdots, m$, $S_{t}^{(i)}$ can be rewritten as
 \begin{equation}
S_t^{(i)}=S_0^{(i)} \exp\left\{\omega_i t+\sum_{j=1}^m\sigma_{ij} \tilde W_t^{(j)} \right\},\quad 0\leq t\leq T,
\label{Q-Stock}　
\end{equation}
where $\omega_i=r-\frac{1}{2}\sum_{j=1}^m\sigma_{ij}^2$.

\begin{remark}
In the Black-Scholes model, the one parameter that cannot be directly observed is the volatility ${\bm \sigma}$ of the risky assets (see \cite{Hull22}). There are two ways to estimate ${\bm \sigma}$. One way is by using the history data of the  risky assets price ${\bm S}$ to estimate ${\bm \sigma}$. The estimate is given by (see \cite{Hull22})
\begin{equation*}
\hat {\bm \sigma}=\frac{1}{\Delta t}\sqrt{\frac{1}{n-1}\sum_{j=1}^n\left( {\bm \varrho}_j  - \bar{\bm \varrho} \right)^2},
\end{equation*}
where ${\bm \varrho}_j:=\ln {\bm S}_j-\ln {\bm S}_{j-1}$, ${\bm S}_j$ is the risky assets price at end of $j$th interval, $n+1$ is the number of observations, and $\Delta t$ is the length of time interval in years. Another way is by using the implied volatility. The implied volatility is the value of ${\bm \sigma}$ that, when substituted into the Black-Scholes-Merton formula (see the equation (\ref{BS-value}) below and \cite{Karatzas91}), gives the observed market price $V$ of European options. 
\end{remark}

Now we assume that the payoff of the simple Asian option is $f(\bar {\bm S}_T)$ (see \cite{Fournie99}), where the payoff function $f:{\mathbb R}^m\rightarrow {\mathbb R}$ is a nonnegative measurable function such that $f(\bar {\bm S}_T)\in L^1(\Omega,{\cal F},{\mathbb Q})$, and 
\begin{equation*}
\bar {\bm S}_T=\frac{1}{T}\int_0^T {\bm S}_t\mathrm{d}t
\end{equation*}
is the continuous-time average of ${\bm S}_t$ over the time interval $[0,T]$. Then the value of  the option at time $0$ is given by (see \cite{Karatzas91})
 \begin{equation}\label{BS-value}　
V= \mathbb{E}_{\mathbb{Q}}\left [ e^{-rT}f(\bar{\bm S}_T)\right].
\end{equation}

The Greek of the option is the derivative (or higher derivative) of $V$ with respect to a model parameter $\alpha$ given by
 \begin{equation}\label{deri}
Greek:=\frac{\partial V}{\partial \alpha}=\frac{\partial}{\partial \alpha}\mathbb{E}_{\mathbb{Q}}\left [ e^{-rT}f(\bar {\bm S}_T)\right ],
\end{equation}
if the above derivative exists.

We only consider $\alpha= S_0^{(l)}$ and $\alpha=\sigma_{lp}$ for any fixed $l\in\{1,\cdots,m\}$ and $p\in\{1,\cdots,m\}$ in this paper. The corresponding Greeks are called $delta$ and $vega$, respectively. When two-order derivative is considered, and $\alpha_1= { S}_0^{(l)}$, $\alpha_2= { S}_0^{(p)}$ for any fixed $l\in\{1,\cdots,m\}$, $p\in\{1,\cdots,m\}$, the corresponding Greek is called $gamma$.

In financial market, $delta$, $gamma$ and $vega$ are the three most important Greeks for an option in practice (see \cite{Hull22}). The $delta$ is often used to hedge the financial derivates. More precisely, the risk in a short position in an option is offset by a delta-hedging strategy of holding $delta$ units of each (underlying) risky asset (see \cite{Glasserman04}). Another use of $delta$ is to calculate the leverage multiple of an option, which is important in the risk management. Similarly, the $gamma$ (resp. $vega$) is used to make the gamma-hedging (resp. vega-hedging) strategy and manage risk as well. We refer to \cite{Hull22} for more details.

\subsection{Greeks for simple Asian options}
\label{greeksimple}
Return to the definition of the simple Asian option Greek (\ref{deri}). If $f$ is Lipschitz continuous, then (\ref{deri}) makes sense, and we can interchange the expectation and differentiation as follows (see \cite{Glasserman04})
 \begin{equation}
\frac{\partial}{\partial \alpha}\mathbb{E}_{\mathbb{Q}}\left [ e^{-rT}f(\bar {\bm S}_T)\right ]=
e^{-rT}\mathbb{E}_{\mathbb{Q}}\left [  \frac{\partial f}{\partial x_l}(\bar {\bm S}_T) \frac{\partial \bar S_T^{(l)}}{\partial \alpha}   \right ] .
\label{formu_greek}
\end{equation}
If we take the partial derivative of $f$ directly in (\ref{formu_greek}) and obtain the relevant Greek formula, the well-known method is called the PW method (see \cite{Glasserman04}). In this case, the estimate of the Greek is dependent on the partial derivative of the payoff function $f$, which implies that this method might be failed when $f$ is not Lipschitz continuous.

In view of the above, we turn to another way, i.e., the Malliavin method, which turns (\ref{formu_greek}) to another form which excludes the partial derivative of $f$.

Since we only consider the expectation here, we identify the filtered probability space $(\Omega, {\cal F},\{{\cal F}_t\}_{0\le t\le T},{\mathbb Q})$ with the space induced by the classical Wiener space in Section \ref{MVcalculus}. 

As for the Malliavin derivatives we are going to use, we have for all $u\in [0,T]$ fixed (see \cite{Xu14,Nualart06}),
 \begin{equation*}
 \begin{aligned}
 &D^{(i)}_t\tilde W^{(j)}_u=\delta_{ij}{\bm 1}\{t\le u\},\quad \quad \quad \quad \quad \ \ i,j=1,\cdots,m;\\
 &{\bm D}_tS^{(j)}_u= (\sigma_{j1}\cdots,\sigma_{jm})'S_u^{(j)}{\bm 1}\{t\le u\} ,\quad j=1,\cdots,m ;\\
 &{\bm D}_t\bar S_u^{(j)}= (\sigma_{j1}\cdots,\sigma_{jm})'\frac{1}{u}\int_t^uS_{\nu}^{(j)}\mathrm{d}\nu ,\quad\hspace*{2pt} j=1,\cdots,m ,
  \end{aligned}
\end{equation*}
where $\bar{\bm S}_u:=\frac{1}{u}\int_0^u {\bm S}_\nu\mathrm{d}\nu$, ${\bm 1}\{\cdot\}$ is the indicator function, and $\delta_{ij}$ denotes the Kronecker delta function.

Applying Theorem \ref{L0} to the right hand of (\ref{formu_greek}) with $n=1$, ${\bm X}=\bar{{\bm S}}_T'$, ${ Y}= \frac{\partial \bar { S}^{(l)}_T}{\partial \alpha}$, $v_t=S_t^{(l)}$, ${\bm c}={\bm\sigma}'$ and $X_t^{(k)}=\frac{1}{T}\int_t^TS^{(k)}_u\mathrm{d}u$ for $k=1,\cdots,m$, we can obtain the Greek formula by Malliavin calculus as follows
\begin{equation}
Greek= \mathbb{E}_{\mathbb Q}\left[e^{-rT}f(\bar {\bm S}_T)   \delta \left( \frac{\partial \bar { S}^{(l)}_T}{\partial \alpha}\frac{2S_t^{(l)}}{T\left( \bar S_T^{(l)}\right)^2}{\bm\sigma}^{-1}_l \right)  \right]= : \mathbb{E}_{\mathbb Q}[e^{-rT}f(\bar {\bm S}_T)\cdot weight],
\label{MV_formu}
\end{equation}
where ${\bm\sigma}^{-1}_l$ denotes the $l$th column of the matrix ${\bm\sigma}^{-1}$, and $weight$ is some random variable independent of the payoff function $f$. Then Lemma \ref{multidelta} can be used to obtain the specific expression of (\ref{MV_formu}). 

Since (\ref{MV_formu}) excludes the partial derivative of $f$, we can extend it to the case that $f$ is not necessarily Lipschitz continuous by the following lemma.  

\begin{lemma}\label{L_O}
\textnormal{(Lemma 12.28,  \cite{Di Nunno09})}
Let $ \alpha\mapsto w_\alpha$ be a process with the parameter $\alpha$ (i.e., $\alpha=S_0^{(l)}$ or $\alpha=\sigma_{lp}$) such that $\alpha\mapsto \mathbb{E}_{\mathbb Q} [w_\alpha^2]<\infty$ is locally bounded. Further assume that
\begin{align}\label{oklemma}
\frac{\partial}{\partial \alpha}{\mathbb E}_{\mathbb Q}[e^{-rT}f(\bar {\bm S}_T)]=e^{-rT}\mathbb{E}_{\mathbb Q}
[f(\bar {\bm S}_T)w_\alpha]
\end{align}
is valid for all $f\in C_c^\infty({\mathbb R}^m)$ (i.e., $f$ is an infinitely differentiable function with compact support). Then (\ref{oklemma}) also holds for all locally integrable function $f: \mathbb{R}^m\rightarrow \mathbb {R}$ such that $f(\bar{\bm S}_T)\in L^2(\Omega,{\cal F},{\mathbb Q})$.
\end{lemma} 

\begin{remark}\label{remarklo}
We can also replace $f(\bar{\bm S}_T)$ by $f( {\bm S}_T,\bar{\bm S}_T)$ in Lemma \ref{L_O} in order to give a similar approximation of the Greeks for complex Asian options in Section \ref{greekcomplex}.
\end{remark}

We conclude the above analysis by the following theorem of the formulae for three common Greeks ($delta$, $gamma$ and $vega$) of simple Asian options.

\begin{theorem}\label{L1}
Let $f(\bar {\bm S}_T) $ be the payoff of the simple Asian option. If  $f $ is locally integrable and $f(\bar{\bm S}_T)\in L^2(\Omega,{\cal F},{\mathbb Q})$, then the formulae for $delta$, $gamma$ and $vega$ are given by
\begin{align*}
&delta: \frac{\partial V}{\partial S_0^{(l)}}=\mathbb{E}_{\mathbb Q}\left[   f(\bar {\bm S}_T)\frac{2e^{-rT}}{S_0^{(l)}}  \tau_l\right];\\
& gamma: \frac{\partial^2 V}{\partial S_0^{(p)}{\partial S_0^{(l)}}}= \mathbb{E}_{\mathbb Q}\Bigg\{ f(\bar {\bm S}_T) \frac{2e^{-rT}}{S_0^{(l)}} \Bigg[   \frac{2}{TS_0^{(p)}\bar S_T^{(p)}} \Bigg(    \tau_l\sum_{i=1}^m\int_0^T\sigma^{-1}_{ip}S_t^{(p)}\mathrm{d}\tilde W_t^{(i)} 
 \\
& \hspace{11em}  -\frac{\delta_{lp} }{ \left(T\bar S_T^{(l)}\right)^2}\sum_{i=1}^m \sigma^{-1}_{il}H_{ilp} +\frac{\tau_l T\bar S_T^{(p)}}{2}   - \frac{\sigma^{-1}_{lp}}{ T\bar S_T^{(l)}}   \sum_{i=1}^m\sigma^{-1}_{il}  \int_0^TS_t^{(l)}S_t^{(p)}\mathrm{d}t   \Bigg)
 -\frac{\delta_{lp}}{S_0^{(l)}}\tau_l \Bigg]  \Bigg\};\\
 &vega: \frac{\partial V}{\partial \sigma_{lp}}=\mathbb{E}_{\mathbb Q} \Bigg [   f(\bar {\bm S}_T) \frac{2e^{-rT}}{\Big(T{{\bar S}^{(i)}_T}\Big)^2} \Big(    \int_0^T S_t^{(l)}g_t\mathrm{d}t    \Big( \tau_l+\frac{1}{2}  \Big) T\bar{S}_T^{(l)}  -\int_0^TS_u^{(l)}  \int_u^T \left(  S_t^{(l)}g_t+S_t^{(l)}\sigma^{-1}_{pl}\right)\mathrm{d}t\mathrm{d}u   \Big) \Bigg],
\end{align*}
\end{theorem}
where $g_t:=\tilde W_t^{(p)}-\sigma_{lp}t$, $\tau_l:=\frac{\sum_{i=1}^m\int_0^T\sigma^{-1}_{il}S_t^{(l)}\mathrm{d}\tilde W_t^{(i)}}{T\bar{S}_T^{(l)}}+\frac{1}{2}$, and
\begin{align*}
H_{ilp}:=T\bar S_T^{(l)}\int_0^TS_u^{(p)}\int_u^TS_t^{(l)}\mathrm{d}\tilde W_t^{(i)}\mathrm{d}u-\int_0^TS_t^{(l)}\mathrm{d}\tilde W_t^{(i)} \int_0^TS_u^{(p)}\int_u^TS_t^{(l)}\mathrm{d}t\mathrm{d}u.
\end{align*}

When $m=1$, from Theorem \ref{L1}, we obtain formulae for $delta$, $gamma$ and $vega$ in one dimensional case in the following corollary.

\begin{corollary} \label{L1C}
Let $m=1$ and $f(\bar { S}_T) $ be the payoff of the simple Asian option. If  $f $ is locally integrable and $f(\bar{ S}_T)\in L^2(\Omega,{\cal F},{\mathbb Q})$, then the formulae for $delta$, $gamma$ and $vega$ are given by
\begin{align*}
&delta: \frac{\partial V}{\partial S_0}=\mathbb{E}_{\mathbb Q}\left[   f(\bar S_T)\frac{2e^{-rT}}{S_0 \sigma^2}\left(   \frac{S_T-S_0}{T\bar S_T}-\omega \right)  \right];\\
& gamma: \frac{\partial^2 V}{\partial S_0^2}=\mathbb{E}_{\mathbb Q}\Bigg[ f(\bar S_T) \frac{4e^{-rT}}{\sigma^4 S_0^2T^2\bar{S}_T^2} \Bigg( S_T^2-2S_TS_0+S_0^2  +\omega rT^2\bar{S}_T^2-2rTS_T\bar{S}_T+2\omega TS_0\bar{S}_T  \Bigg) \Bigg];\\
& vega:\frac{\partial V}{\partial \sigma}=\mathbb{E}_{\mathbb Q}\Bigg[   f(\bar S_T) \frac{2e^{-rT}}{\sigma^2 T^2 \bar{S}_T^2}    \Bigg(    \left(S_T-S_0-\left(r-\sigma^2\right)T\bar{S}_T\right)\int_0^TS_tg_t\mathrm{d}t    - \sigma^2\int_0^TS_u\int_u^T S_t g_t\mathrm{d}t\mathrm{d}u-\frac{\sigma T^2}{2}\bar{S}_T^2    \Bigg)   \Bigg].
\end{align*} 
\end{corollary}

The formulae for $delta$ and $gamma$ in Theorem \ref{L1} can also be found in \cite{Xu14}. However, our method is more concise and general since all the results are based on a general integration by parts formula in Theorem \ref{L0} rather than the complicated calculation for each specific example in \cite{Xu14}. Moreover, Theorem \ref{L0} could also be allowed to calculate the Greeks formulae for complex Asian options in the next subsection. 
 
 \subsection{Greeks for complex Asian options}
\label{greekcomplex}
Now we consider the payoff of the complex Asian option $f({\bm S}_T,\bar{\bm S}_T)$, where the payoff function $f: {\mathbb R}^{2m}\rightarrow {\mathbb R}$ is a nonnegative measurable function such that $f({\bm S}_T,\bar{\bm S}_T)\in L^1(\Omega,{\cal F},{\mathbb Q})$.

Similar to the case of the simple Asian option, if $f$ is Lipschitz continuous, we can obtain the Greek of the complex Asian option as follows
\begin{equation}\begin{aligned}
Greek&=\frac{\partial}{\partial \alpha}{\mathbb E}_{\mathbb Q}[e^{-rT}f({\bm S}_T,\bar{\bm S}_T)]\\
&=e^{-rT}{\mathbb E}_{\mathbb Q}\left[ \frac{\partial f}{\partial x_{1l}}({\bm S}_T,\bar{\bm S}_T)\frac{\partial S^{(l)}_T}{\partial \alpha}+\frac{\partial f}{\partial x_{2l}}({\bm S}_T,\bar{\bm S}_T)\frac{\partial \bar{S}_T^{(l)}}{\partial \alpha} \right].
\label{greekofcomp}
\end{aligned}\end{equation}

However, when dealing with the complex Asian option, the choice of ${\bm v}_t $ in Theorem \ref{L0} is important since a good choice can simplify our formulae greatly. We introduce the set $\Lambda$ defined by
\begin{equation*}
\Lambda:=\{  a\in  L^2([0,T]) :\int_0^Ta_t\mathrm{d}t=0, \int_0^Ta_t\int_0^tS_u^{(k)}\mathrm{d}u\mathrm{d}t\neq 0,\forall k=1,\cdots,m \}.
\end{equation*}

 Let $a\in\Lambda$. Then applying Theorem \ref{L0} to the right hand of (\ref{greekofcomp}) with $n=2$, ${\bm X}_1={\bm S}_T'$, ${\bm X}_2=\bar{\bm S}_T'$, ${\bm Y}=(\frac{\partial S_T^{(l)}}{\partial \alpha}, \frac{\partial \bar{S}_T^{(l)}}{\partial \alpha})$, ${\bm v}_t=(a_t,S_t^{(l)})$, ${\bm c}={\bm \sigma}'$, $X_t^{(1k)}=S_T^{(k)}$ and $X_t^{(2k)}=\frac{1}{T}\int_t^TS_u^{(k)}\mathrm{d}u$ for $k=1,\cdots,m$, we can obtain the matrix
 \begin{equation}
 {\bm\zeta}=
 \left(
      \begin{array}{cc}
       0 &-\frac{1}{T}\int_0^Ta_t\int_0^tS_u^{(l)}\mathrm{d}u\mathrm{d}t \\
       TS_T^{(l)}\bar S_T^{(l)} &\frac{T}{2}\left(\bar S_T^{(l)}\right)^2
      \end{array}
      \right)   ,      
 \end{equation}
 and the Greek formula
\begin{equation}\label{MV_formu2}\begin{aligned}
& Greek \\
=& \mathbb{E}_{\mathbb Q}\Bigg[e^{-rT}f({\bm S}_T,\bar {\bm S}_T)   \delta \Bigg(\Bigg( \frac{\frac{T}{2}\bar S_T^{(l)}\frac{\partial S_T^{(l)}}{\partial \alpha}a_t-TS_T^{(l)}\frac{\partial\bar S_T^{(l)}}{\partial \alpha}a_t}{S_T^{(l)}\int_0^Ta_\nu\int_0^\nu S_u^{(l)}\mathrm{d}u\mathrm{d}\nu} 
+\frac{\partial S_T^{(l)}}{\partial \alpha}\frac{S_t^{(l)}}{TS_T^{(l)}\bar S_T^{(l)}}  \Bigg){\bm\sigma}^{-1}_l\Bigg)  \Bigg]\\
=&: \mathbb{E}_{\mathbb Q}[e^{-rT}f({\bm S}_T,\bar {\bm S}_T)\cdot weight].
\end{aligned}\end{equation}
By Lemma \ref{L_O} and Remark \ref{remarklo}, the formula (\ref{MV_formu2}) still holds when $f$ is locally integrable and $f({\bm S}_T,\bar{\bm S}_T)\in L^2(\Omega,{\cal F},{\mathbb Q})$. We conclude the above analysis by the following theorem of the formulae for the Greeks for complex Asian options.

\begin{theorem}\label{L12}
Let $f({\bm S}_T,\bar {\bm S}_T) $ be the payoff of the complex Asian option, and $a\in \Lambda$. If $f $ is locally integrable and $f({\bm S}_T,\bar{\bm S}_T)\in L^2(\Omega,{\cal F},{\mathbb Q})$, then the formulae for $delta$, $gamma$ and $vega$ are given by
\begin{align*}
&delta: \frac{\partial V}{\partial S_0^{(l)}}=\mathbb{E}_{\mathbb Q}\Bigg[  f({\bm S}_T,\bar{\bm S}_T)\frac{e^{-rT}}{S_0^{(l)}} \Bigg( \frac{1}{T\bar S_T^{(l)}} \sum_{i=1}^m\int_0^T \sigma^{-1}_{il} S_t^{(l)} \mathrm{d}\tilde W^{(i)}_t  +  \frac{\int_0^Ta_t\int_t^TS_u^{(l)}\mathrm{d}u\mathrm{d}t }{2h_T}   \\
&\hspace{8em}+\frac{1}{2} - \frac{T\bar S^{(l)}_T\sum_{i=1}^m\int_0^T\sigma^{-1}_{il}a_t\mathrm{d}\tilde W^{(i)}_t}{2h_T}    -\frac{T\bar S_T^{(l)}\int_0^Ta_t\int_t^Ta_\nu \int_t^\nu S_u^{(l)}\mathrm{d}u\mathrm{d}\nu\mathrm{d}t}{2h_T^2} \Bigg)   \Bigg] ;\\
& gamma: \frac{\partial^2 V}{\partial S_0^{(p)}{\partial S_0^{(l)}}}= \mathbb{E}_{\mathbb Q} \Bigg[f({\bm S}_T,\bar {\bm S}_T)  \frac{e^{-rT}}{S_0^{(p)}S_0^{(l)}}  \delta \Bigg({\bm\sigma}^{-1}_l\Bigg(  \frac{\frac{T}{2}\bar S_T^{(l)}w_1a_t-TS_T^{(l)}w_2a_t}{S_T^{(l)}h_T} +  \frac{\frac{T}{2}\bar S_T^{(l)}w_1a_t-TS_T^{(l)}w_2a_t}{S_T^{(l)}h_T}\Bigg)\Bigg)\Bigg];\\
& vega:  \frac{\partial V}{\partial \sigma_{lp}}=\mathbb{E}_{\mathbb Q} \Bigg[  f({\bm S}_T,\bar{\bm S}_T) e^{-rT}\Bigg(     
 \frac{T\bar S_T^{(l)}g_T}{2h_T} \sum_{i=1}^m\int_0^T\sigma^{-1}_{il}a_t\mathrm{d}\tilde W_t^{(i)}  +\frac{ g_T}{ 2} -\sigma_{pl}^{-1} \\
 &\hspace{8em}     +\frac{T\bar S_T^{(l)}g_T}{2 h_T^2}\int_0^T a_t \int_t^Ta_\nu\int_t^\nu S_u^{(l)}\mathrm{d}u   \mathrm{d}\nu\mathrm{d}t  -\frac{1}{2h_T}\int_0^T\Big(   g_T a_t\int_t^T S_u^{(l)}\mathrm{d}u+T\bar S_T^{(l)}   \sigma^{-1}_{pl} a_t \Big)   \mathrm{d}t \\
 &\hspace{8em}  - \frac{S_T^{(l)}\int_0^TS_t^{(l)}g_t\mathrm{d}t}{S_T^{(l)}h_T}\sum_{i=1}^m\int_0^T \sigma^{-1}_{il}a_t \mathrm{d}\tilde W^{(i)}_t-\frac{\int_0^TS_t^{(l)}g_t\mathrm{d}t}{h_T^2 }
\int_0^Ta_t\int_t^Ta_\nu\int_t^\nu S_u^{(l)}\mathrm{d}u\mathrm{d}\nu\mathrm{d}t\\
&\hspace{8em}   +\frac{1}{h_T} \int_0^T a_t     \int_t^T \Big( S_u^{(l)}g_u   +\sigma^{-1}_{pl}   S_u^{(l)}   \Big)\mathrm{d}u   \mathrm{d}t  +\frac{g_T}{T\bar S_T^{(l)}}\sum_{i=1}^m\int_0^T\sigma^{-1}_{il} S_t^{(l)}\mathrm{d}\tilde W^{(i)}_t 
  \Bigg)\Bigg] ,
\end{align*}
where $h_t:=\int_0^t a_\nu\int_0^\nu S_u^{(l)}\mathrm{d}u\mathrm{d}v$, $\tilde h_t:=\int_0^t a_\nu\int_0^\nu S_u^{(p)}\mathrm{d}u\mathrm{d}v$, $g_t:=\tilde W_t^{(p)}-\sigma_{lp}t$, and 
\begin{align*}
w_1=&-\frac{T\bar S_T^{(p)}S_T^{(l)}}{2\tilde h_T}\sum_{i=1}^m\int_0^T\sigma^{-1}_{ip}a_t\mathrm{d}\tilde W^{(i)}_t-\delta_{lp}S_T^{(l)}+\frac{1}{2}+\frac{S_T^{(l)}}{2\tilde h_T}\int_0^T a_t\int_t^TS_u^{(p)}\mathrm{d}u\mathrm{d}t  -\frac{T\bar S_T^{(p)}S_T^{(l)}}{2\tilde h_T^2}\int_0^T a_t\int_t^Ta_\nu  \int_t^\nu S_u^{(p)}\mathrm{d}u\mathrm{d}\nu\mathrm{d}t  
\\&+\frac{S_T^{(l)}}{T\bar S_T^{(p)}}\sum_{i=1}^m\int_0^T\sigma^{-1}_{ip}S_t^{(p)}\mathrm{d}\tilde W^{(i)}_t;\\
w_2=&-\frac{T\bar S_T^{(p)}\bar S_T^{(l)}}{2\tilde h_T}\sum_{i=1}^m\int_0^T\sigma^{-1}_{ip}a_t\mathrm{d}\tilde W_t^{(i)}+\frac{\bar S_T^{(l)}}{2\tilde h_T}\int_0^Ta_t\int_t^TS_u^{(p)}\mathrm{d}u\mathrm{d}t-\frac{\delta_{lp} \bar S_T^{(p)}}{2}+\frac{\bar S_T^{(l)}}{2}  +\delta_{lp}\frac{\bar S_T^{(p)}}{2\tilde h_T}\int_0^Ta_t\int_t^TS_u^{(l)}\mathrm{d}u\mathrm{d}t\\
&   -\frac{T\bar S_T^{(p)}\bar S_T^{(l)}}{2\tilde h_T^2}\int_0^Ta_t\int_t^T a_\nu\int_t^\nu S_u^{(p)}\mathrm{d}u\mathrm{d}\nu\mathrm{d}t +\frac{\bar S_T^{(l)}}{T\bar S_T^{(p)}}\sum_{i=1}^m\int_0^T\sigma^{-1}_{ip}S_t^{(p)}\mathrm{d}\tilde W_t^{(i)}.
\end{align*}
\end{theorem}

\begin{remark}
If $f({\bm x}_1,{\bm x}_2)$ is only dependent on ${\bm x}_1$ in Theorem \ref{L12}, namely, $f({\bm S}_T,\bar{\bm S}_T)=f(\bar{\bm S}_T)$, we can obtain the Greeks formulae for European options.
\end{remark}
 
When $m=1$,  we can choose $a_t=\frac{T}{2}-t$ in Theorem \ref{L12} to obtain the specific expression of the formulae in one dimensional case. For example, the formula for $delta$ is given in the following corollary.

\begin{corollary}
Let $m=1$ and $f({ S}_T,\bar {S}_T) $ be the payoff of the complex Asian option. If $f $ is locally integrable and $f({ S}_T,\bar{ S}_T)\in L^2(\Omega,{\cal F},{\mathbb Q})$, then the formula for $delta$ is given by
\begin{equation*}\begin{aligned}
delta: \frac{\partial V}{\partial S_0}=\mathbb{E}_{\mathbb Q}\Bigg[ & f({ S}_T,\bar{ S}_T)\frac{e^{-rT}}{S_0} \Bigg( \frac{1}{\sigma^2 T\bar S_T} \left( S_T-S_0-rT\bar S_T \right) - \frac{T\bar S_T}{2\sigma h_T} \Big(\int_0^T\tilde W_t\mathrm{d}t-\frac{T}{2}\tilde W_T\Big)+  \frac{\int_0^T \int_t^TS_u(\frac{T}{2}-t)\mathrm{d}u\mathrm{d}t }{2h_T}    \\
&  +\frac{1}{2}  -\frac{T\bar S_T\int_0^T\int_t^T   \int_t^\nu S_u(\frac{T}{2}-\nu)(\frac{T}{2}-t)\mathrm{d}u\mathrm{d}\nu\mathrm{d}t}{2h_T^2}       \Bigg)   \Bigg] .
\end{aligned}\end{equation*}
\end{corollary}

Lacking closed form solutions for the expectations (integrals) in Theorems \ref{L1} and \ref{L12}, the Greeks above could only be calculated by numerical methods. The rest of this paper is focused on developing a QMC-based method to calculate the expectations in Theorems \ref{L1} and \ref{L12} more efficiently.

\section{Quasi-Monte Carlo method}
\label{Sec-QMC}
 MC and QMC are two important methods in numerical simulation to calculate high-dimensional integral over the cube $[0,1)^d$,
\begin{equation*}
I(F)=\int_{(0,1]^d}F({\bm u})\mathrm{d}{\bm u}=\mathbb{E}\left[  F({\bm U})  \right] ,
\end{equation*}
where $\bm U$ is a $d$-dimensional uniform random variable in $(0,1]^d$. Compared with the MC method, QMC uses deterministic points, which are more uniformly distributed over $(0,1]^d$ instead of  random points. The QMC estimate of $I(F)$ can be written as 
\begin{equation*}
Q_n(F)=\frac{1}{n}\sum_{i=1}^nF({\bm u}_i),
\end{equation*}
where ${\bm u}_1,\cdots,{\bm u}_n$ are deterministic low discrepancy points. The QMC error can be bounded by the Koksma-Hlawka inequality
\begin{equation}
|I(F)-Q_n(F)|\le V_{\text{HK}}(F)D^{*}({\bm u}_1,\cdots,{\bm u}_n),
\label{KH}
\end{equation}
where $V_{\text{HK}}(F) $ is the variation of $F$ in the sense of  the paper of Hardy and  Krause (see \cite{Niederreiter92}), and $D^{*}({\bm u}_1,\cdots,{\bm u}_n)$ is the star discrepancy. This inequality asserts that the convergence rate of QMC for functions of finite variation is $O(n^{-1}(\ln n)^d)$, which is asymptotically much better than the MC rate $O(n^{-1/2})$. We can choose many low discrepancy sequences to achieve this rate, such as the Sobol' sequence (see \cite{Sobol67}), the Faure sequence (see \cite{Faure82}) and the Halton sequence (see \cite{Halton60}).  

Despite the advantage of QMC over MC in asymptotic convergence rate, there are two crucial factors limiting the efficiency of QMC in practice, namely, the smoothness of the integrand and the high dimensionality (see \cite{Morokoff95}).

Although high dimensionality is a big challenge for QMC, it may still provide better performance if the function has low effective dimension (see \cite{Caflisch97,Wang03}). In order to enhance the efficiency of QMC, some path generation methods (PGMs) are developed to reduce the effective dimension. To introduce PGM, we consider the following expectation in mathematical finance
\begin{equation}
I(g)=\mathbb{E}\left[g({\bm X})\right]
\label{PGM1}
\end{equation}
for some function $g$, where ${\bm X}\sim N({\bm 0}_d,{\bm\Sigma})$ and ${\bm \Sigma}$ is the covariance matrix. The essence of a PGM is to generate the samples of $\bm X$ as ${\bm X}={\bm A}{\bm Y}$, where ${\bm Y}\sim N({\bm 0}_d,{\bm I}_d)$ and the matrix $\bm A$ satisfies ${\bm A}{\bm A}'={\bm \Sigma}$. This is equivalent to transform the integral (\ref{PGM1}) as 
\begin{equation}
I(g)=\mathbb{E}\left[g({\bm X})\right]=\mathbb{E}\left[g({\bm AY	})\right].
\label{PGM2}
\end{equation}
The choices of $\bm A$ make no difference in MC setting from the point of view of variance (because the variance is unchanged), whereas different choices of $\bm A$ may lead to different efficiency in QMC setting due to the possible change of the effective dimension and the smoothness properties. Classical methods of PGM include standard (STD) construction, Brownian bridge (BB) construction and principal components analysis (PCA) construction (see \cite{Glasserman04}). Many new methods are  proposed to reduce the effective dimension, such as the linear transform (LT) construction (see \cite{Imai04}), the  clustering QR (CQR) construction (see \cite{Weng17}) and the GPCA construction (see \cite{Xiao19}). The GPCA construction performs best in many cases, especially when the integrand has complicated structure, since it uses the structure and information of the integrand (see \cite{Xiao18}). 

Bad smoothness is another factor that could seriously decrease the efficiency of the QMC method. Strenuous efforts are devoted to conquering it, such as the orthogonal transformation (OT) method (see \cite{Wang13}), the QR method (see \cite{He14}), and the CQMC method (see \cite{Xiao18}). 

In practice, the Koksma-Hlawka inequality is not very useful although it asserts that QMC provides a better asymptotic convergence rate. The first reason is that it is hard to calculate the variation $V_{\text{HK}}(f)$. The second reason is that many functions in finance have unbounded variation. Moreover, since the QMC points are deterministic, we need to randomize QMC points to obtain error estimation. Randomized quasi-Monte Carlo (RQMC) methods are proposed to solve those problems, such as the scrambling, the digital shift, and the random shift (see \cite{Owen98,Matousek98,LEcuyer05}).


\section{QMC-based conditional Malliavin method for Asian option Greeks}
\label{Sec:QMC-CMV}
We focus on calculating the following expectation
\begin{equation}
I(g)=\mathbb{E}\left[ g({\bm X}) \right],
\label{4-1}
\end{equation}
where ${\bm X}=(X_1,\cdots,X_d)'$ is a $d$-dimensional random vector with independent components, and $g({\bm X})$ is integrable. Let $\rho_k(x_k)$ be the density function of $X_k$, and $\rho({\bm x}_u)=\Pi_{k\in u}\rho_k(x_k)$ for $u\subseteq \{1,\cdots, d\}$.

\subsection{Conditional QMC method } 
\label{Sec:CQMC}
CQMC method is a smoothing method similar to the CMC method (see \cite{Fu97}). We describe the main idea of CQMC method below. We refer to \cite{Xiao18} for details.

Suppose that the function $g$ satisfies the so-called variable separation condition 
\begin{equation}
g({\bm x})=h({\bm x}){\bm 1}\{  \psi_d({\bm z})<x_j<\psi_u({\bm z})  \},
\label{4-2}
\end{equation}
for some $x_j$, where $h$ is a continuous function and $\psi_d$ and $\psi_u$ are functions of ${\bm z}:=(x_1,\cdots,x_{j-1},x_{j+1},\cdots,x_d)'$. Then the expectation (\ref{4-1}) can be written as
\begin{equation*}
I(g)=\mathbb{E}\left[ g({\bm X}) \right] =\int_{\mathbb{R}^{d-1}}\left(    \int_{\psi_d}^{\psi_u} h({\bm x})\rho_j(x_j)\mathrm{d}x_j  \right){\bm 1}\{\psi_d({\bm z})<\psi_u({\bm z})\}\rho({\bm z})\mathrm{d}{\bm z} .
\end{equation*}
Suppose that the conditional expectation
\begin{equation}
\mathbb{E}\left[ g({\bm X}) | {\bm Z}\right] =    \int_{\psi_d}^{\psi_u} h({\bm x})\rho_j(x_j)\mathrm{d}x_j  {\bm 1}\{\psi_d({\bm z})<\psi_u({\bm z})\}=:G({\bm Z})
\label{4-5}
\end{equation}
is analytically available, then we have
\begin{equation}
I(g)=\mathbb{E}\left[ g({\bm X}) \right] =\mathbb{E}\left[ \mathbb{E}\left[ g({\bm X}) | {\bm Z}\right]     \right] =\mathbb{E}\left[ G({\bm Z}) \right].
\label{4-6}
\end{equation}
Therefore, $G({\bm Z})$ is a new estimate of  $I(g)$, which is called the conditional estimate.

Many functions in financial engineering satisfy the separation of variables condition (see the expression (\ref{functionofg}) for the function $g$ of Asian option Greeks in the next subsection). The method of conditional expectation above is a form of CMC to reduce the variance. The idea can also be used in QMC setting if we use the QMC points instead of MC points in calculating ${\mathbb E}[G({\bm Z})]$, and we call it CQMC method.

It is well-known that the performance of QMC integration strongly depends on the smoothness of the integrand. If the integrand possesses excellent smoothness, the calculation efficiency may be increased significantly. Furthermore, the effective dimension reduction method (such as the GPCA method) can be used to achieve better results. 

The next theorem can be proved similarly as the Theorem A.1 in \cite{Zhang20}. Therefore, we omit the proof of Theorem \ref{THM:CMC}. It shows the good smoothness of the new function $G$. Before that we need to introduce a definition of  uniform convergence.

 \begin{definition}
We say the integral $\int \xi({\bm y},{\bm x})\mathrm{d}{\bm x}$ converges uniformly on a set $\cal O$, if for any $\varepsilon>0$ and ${\bm y}\in {\cal O}$, there exists a constant $A$ independent of ${\bm y}$ such that $\left|   \int_{\Vert {\bm x}\Vert>A} \xi ({\bm y},{\bm x})\mathrm{d}{\bm x} \right|<\varepsilon$.

\end{definition}

Note that here ${\bm y}$ and ${\bm x}$ can be scalars or vectors, and $\Vert {\bm x}\Vert$ represents the Euclidean norm of ${\bm x}$.

\begin{theorem}
Suppose that function $g$ has the form in (\ref{4-2}) and $G$ is defined by (\ref{4-5}). Assume that the following conditions are satisfied
\begin{itemize}
 \item[(a)]   $\psi_d$ and $\psi_u$ are continuous;
\item[(b)]   $h$ is nonnegative.
\end{itemize}
 Then the following two statements are equivalent
 \begin{itemize}
 \item[(c)] If one of $\psi_u$ and $\psi_d$ is infinite, then the integral $\int_{\mathbb{R}} h({\bm x})\rho_j(x_j)\mathrm{d}x_j$ converges uniformly on a compact neighborhood of ${\bm z}$, where ${\bm z}:=(x_1,\cdots,x_{j-1},x_{j+1},\cdots,x_d)'$ for some $j$;
 \item[(d)] $G$ is continuous.
\end{itemize}
\label{THM:CMC} 
\end{theorem}

Theorem \ref{THM:CMC} gives conditions for which $G$ is a continuous function. For most common Asian options, condition (a) is hold. Note that condition (c) is still sufficient for (d) even without the condition (b). Thus if both $\psi_d$ and $\psi_u $ are finite, then $G$ is continuous, and if one of them is infinite, we need an extra condition of uniform convergence. The next lemma, which is similar to the Lemma 4.2 in \cite{Zhang20}, shows that the assumption of uniform convergence is also satisfied for most common Asian options.
 
\begin{lemma}\label{L3}
For fixed $j$, assume that $h({\bm x})$ has the form $h({\bm x})=\sum_{k=0}^n A_kx_j^ke^{Bx_j}-C$ for some positive integer $n$, where $A_k$, $B$ and $C$ are functions of ${\bm z}$ (given in Theorem \ref{THM:CMC}), bounded on a compact set $\Lambda$ but independent of $x_j$. Then $\int h({\bm x})\phi(x_j)\mathrm{d}x_j$ converges uniformly on $\Lambda$, where $\phi$ is the density function of the standard normal variable $x_j$.
\end{lemma}

By Lemma \ref{L3} we can easily verify the uniform convergence condition (c) in Theorem \ref{THM:CMC}. Fortunately, in common cases the function $h$ has the form given in Lemma \ref{L3} (see Sections~\ref{Sub_ill} for the following concrete examples: binary Asian options, Asian call options and up-and-out Asian call options).  

 \subsection{QMC-based conditional Malliavin method for calculating Greeks} 
\label{Sub:QMC-CMV}
For convenience, from now on we only consider the Greeks for simple Asian options in the one-dimensional case (hereafter we omit the `simple'). However, all results in this subsection can be extended to the general situation. 

We presented the formulae for Asian option Greeks by Malliavin calculus in Corollary \ref{L1C}. For $delta$, $gamma$ and $vega$, the Geeks have the following expression
\begin{equation}
\label{greek_f}
\begin{aligned}
&Greek \\
=&\mathbb{E}_{\mathbb Q}\left[e^{-rT}f(\bar S_T)\cdot weight\left(\bar S_T,S_T,\int_0^TS_t\left( \tilde W_t-\sigma t   \right )\mathrm{d}t, \int_0^TS_u\int_u^T S_t\left(\tilde W_t-\sigma t\right)\mathrm{d}t\mathrm{d}u  \right)\right]\\
=&:\mathbb{E}_{\mathbb Q}[ \hat\theta]\\
=&:\theta.
\end{aligned}
\end{equation}

In the numerical simulation, we usually take MC methods to estimate $\theta$, and this requires us to discretize the continuous path first. We may choose equidistant points $0=t_0<t_1<\cdots<t_d=T$ between time $0$ and the maturity date $T$, at which we get the price of  underlying asset $S_{t_j}$ in discrete times. For simplicity, we denote $S_{t_j}$ by $S_j$. Based on the expression (\ref{qstock}) of $S$ under the risk-neutral probability measure $\mathbb{Q}$, we have
  \begin{equation}
\begin{aligned}
S_j=S_0\exp\left( \omega t_j+\sigma \tilde{W}_{t_j}  \right),\quad 0\le t\le T.
\end{aligned}
\label{BS-SJ}
\end{equation}
To simulate the continuous-time average $\bar S_T=\frac{1}{T}\int_0^T S_t\mathrm{d}t$, let
\begin{equation*}
S_A:=\frac{1}{d}\sum_{j=1}^d S_j
\end{equation*}
be the discrete average of the asset prices. Then by the theory of Riemann integral, we have the pointwise convergence as follows
\begin{equation*}
\lim_{d\rightarrow +\infty}S_A=\bar S_T.
\end{equation*}
It is natural to take the following simulation estimate $\hat\theta_d$ to approximate $\theta$ in (\ref{greek_f})
\begin{equation}
\label{greek_f2}
\hat\theta_d=e^{-rT}f( S_A)\cdot weight\left( S_A,S_d,   \frac{T}{d}\sum_{j=1}^d S_j(\tilde W_{t_j}-\sigma t_j)  , \frac{T^2}{d^2}\sum_{i=1}^d\sum_{j=i}^d S_iS_j(\tilde W_{t_j}-\sigma t_j)   \right).
\end{equation}

Although in most literature $\hat\theta_d$ is treated as identical to $\hat\theta$ directly, we can not ignore the fact that $\hat \theta_d$ is a biased estimate of $\theta$. Fortunately, we can prove that $\hat\theta_d$ is an asymptotic unbiased estimate with respect to $d$ when $f$ is a continuous function with linear growth.

\begin{theorem}
 Let $f(\bar S_T) $ be the payoff of the Asian option. If $f$ is continuous with linear growth, then $\hat\theta_d$ is an asymptotic unbiased estimate of $\theta$ with respect to $d$, i.e.,
 \begin{equation*}
  \lim_{d\rightarrow +\infty}\mathbb{E}_{\mathbb Q}  [\hat\theta_d  ] =\mathbb{E}_{\mathbb Q}  [\hat\theta ]=\theta.
  \end{equation*}
\label{THM:bias} 
\end{theorem}
\begin{proof}
From the formulae of Greeks in Corollary \ref{L1C}, we only need to consider the convergence of the following expectations
\begin{equation*}\begin{aligned}
&\mathbb{E}_{\mathbb Q}\left[  f(S_A)  \frac{S_d^{m_1}}{S_A^{m_2}} \right], \\ 
&\mathbb{E}_{\mathbb Q}\left[  f(S_A)  \frac{S_d^{m_1}}{S_A^{m_2}} \frac{1}{d}\sum_{j=1}^d S_j(\tilde W_{t_j}-\sigma t_j)   \right],  \\ 
&\mathbb{E}_{\mathbb Q}\left[  \frac{f(S_A)}{S_A^2} \frac{1}{d^2}\sum_{i=1}^d\sum_{j=i}^d S_iS_j(\tilde W_{t_j}-\sigma t_j) \right], 
\end{aligned}\end{equation*}
where $m_1$ and $m_2$ are nonnegative integers.

Since $f$ is continuous, we have $\lim_{d\rightarrow +\infty}f(S_A)=f(\bar S_T)$. Thus each integrand in the above expectations are convergent pointwisely. By dominated convergence theorem, we only need to consider the upper and lower bounds of each integrand.

For the first integrand, since $f$ is nonnegative with linear growth, we have
 \begin{equation*}
 \begin{aligned}
0\le   f(S_A)  \frac{S_d^{m_1}}{S_A^{m_2}}\le C(1+S_A)  \frac{S_d^{m_1}}{S_A^{m_2}}\le  C\left(1+\max_{0\le t\le T}S_t\right)\frac{S_d^{m_1}}{(\min_{0\le t\le T}S_t  )^{m_2}  },
\end{aligned}
\end{equation*}
where $C$ is a positive constant.

For the second integrand, we have
\begin{equation*}
 \begin{aligned}
   f(S_A)  \frac{S_d^{m_1}}{S_A^{m_2}} \frac{1}{d}\sum_{j=1}^d S_j(\tilde W_{t_j}-\sigma t_j) 
&\ge f(S_A)  \frac{S_d^{m_1}}{S_A^{m_2}}\min_{0\le t\le T}S_t(\tilde W_t-\sigma t)\\
  &\ge -   C\left(1+\max_{0\le t\le T}S_t\right)\frac{S_d^{m_1}}{(\min_{0\le t\le T}S_t  )^{m_2}  }    \left[\min_{0\le t\le T}S_t(\tilde W_t-\sigma t)\right]^-,
\end{aligned}
\end{equation*}
and
\begin{equation*}
 \begin{aligned}
   f(S_A)  \frac{S_d^{m_1}}{S_A^{m_2}} \frac{1}{d}\sum_{j=1}^d S_j(\tilde W_{t_j}-\sigma t_j)  
 &\le  f(S_A)  \frac{S_d^{m_1}}{S_A^{m_2}}\max_{0\le t\le T}S_t(\tilde W_t-\sigma t) \\
  &\le     C\left(1+\max_{0\le t\le T}S_t\right)\frac{S_d^{m_1}}{(\min_{0\le t\le T}S_t  )^{m_2}  }    \left[\max_{0\le t\le T}S_t(\tilde W_t-\sigma t)\right]^+,
\end{aligned}
\end{equation*}
where $[\cdot]^+$ and $[\cdot]^-$ denotes the positive part and the negative part of a function, respectively.

For the third integrand, we have
\begin{equation*}
 \begin{aligned}
 \frac{f(S_A)}{S_A^2} \frac{1}{d^2}\sum_{i=1}^d\sum_{j=i}^d S_iS_j(\tilde W_{t_j}-\sigma t_j)&\ge
  \frac{f(S_A)}{S_A^2} \frac{d^2+d}{2d^2}  \min_{0\le u\le T\atop u\le t\le T}S_uS_t(\tilde W_{t}-\sigma t)\\
  &\ge - \frac{ C\left(1+\max_{0\le t\le T}S_t\right)}{\left(    \min_{0\le t\le T} S_t\right)^2}   \left[ \min_{0\le u\le T\atop u\le t\le T}S_uS_t(\tilde W_{t}-\sigma t)\right]^-,
 \end{aligned}
\end{equation*}
and
\begin{equation*}
 \begin{aligned}
 \frac{f(S_A)}{S_A^2} \frac{1}{d^2}\sum_{i=1}^d\sum_{j=i}^d S_iS_j(\tilde W_{t_j}-\sigma t_j)&\le
  \frac{f(S_A)}{S_A^2} \frac{d^2+d}{2d^2}  \max_{0\le u\le T\atop u\le t\le T}S_uS_t(\tilde W_{t}-\sigma t)\\
  &\le  \frac{ C\left(1+\max_{0\le t\le T}S_t\right)}{\left(    \min_{0\le t\le T} S_t\right)^2}   \left[ \max_{0\le u\le T\atop u\le t\le T}S_uS_t(\tilde W_{t}-\sigma t)\right]^+.
 \end{aligned}
\end{equation*}

Thus the conclusion follows.
\end{proof}

From the above proof, we can also get the following result in the Asian options pricing.
\begin{corollary}
 Let $f(\bar S_T) $ be the payoff of the Asian option. If $f$ is continuous with linear growth, then we have 
 \begin{equation*}
 \lim_{d\rightarrow +\infty}\mathbb{E}_{\mathbb Q}  [f(S_A)  ] =\mathbb{E}_{\mathbb Q}  [f(\bar S_T) ],
 \end{equation*}
 where $S_A=\frac{1}{d}\sum_{j=1}^d S_j$ is the discrete average of the asset prices.
\end{corollary}

It is open about the bias or the asymptotic unbiasedness of $\hat\theta_d$ when $f$ is not continuous. Nevertheless, we still use $\hat\theta_d$ to estimate $\theta$ in practice.

Note that $\mathbb{E}_{\mathbb Q}[\hat\theta_d]$ can be expressed as
\begin{equation}\label{functionofg}
\mathbb{E}_{\mathbb Q}[\hat\theta_d]=\mathbb{E}_{\mathbb Q}[g({\bm X})],
\end{equation}
 where ${\bm X}=(X_1,\cdots,X_d)'$ is a $d$-dimensional standard normal random vector under the risk-neutral probability measure $\mathbb{Q}$ (see the next section). Thus the CQMC method can be applied to calculate the Greeks.

 We defined $G({\bm Z})={\mathbb E}_{\mathbb Q}[g({\bm X})|{\bm Z}]$  the conditional Malliavin (CMV) estimate of the Greeks in Section \ref{Sec:CQMC}. If the CMV estimate $G({\bm Z})$ is analytically available, we can use QMC method to estimate the Greeks as follows 
\begin{equation}
\mathbb{E}_{\mathbb Q} [g({\bm X})]=\mathbb{E}_{\mathbb Q} [G({\bm Z})] \approx \frac{1}{n}\sum_{i=1}^n G(\Phi^{-1}({\bm u}_i)),
\label{4-6}
\end{equation}
where ${\bm u}_1,\cdots,{\bm u}_n$ are $(d-1)$-dimensional RQMC points, and $\Phi$ is the distribution function of ${\bm Z}$. This is QMC-based CMV method. We call it QMC-CMV method for short.

Note that we only consider the continuous-time Asian option Greeks in this paper. For the discrete-time Asian option Greeks, the Malliavin estimate degenerates into the likelihood ratio (LR) estimate, and there are many powerful methods based on the LR method and the pathwise (PW) method combined with CQMC method to improve the smoothness. We refer to \cite{Xiao18,Zhang20} for more details.

\section{QMC-CMV method: implementation and examples}
\label{Sub_ill}
 In this section, we show the implementation of the QMC-CMV method, and give the CMV estimates for the Greeks of three common Asian option Greeks (i.e., binary Asian options, Asian call options and up-and-out Asian call options) as illustrations.
 
By discretization, we get the price of  underlying asset $S_{j}$ at time  $t_j$. To separate variables, let 
 \begin{equation}
\begin{aligned}
\tilde{S}_j&=S_0\exp\left(\omega(t_j-t_1)+\sigma(\tilde{W}_{t_j}-\tilde{W}_{t_1})\right)=S_0\exp\left(  \omega(t_j-t_1)+\sigma \bar{W}_{t_j-t_1}  \right),
\end{aligned}
\label{BS-tildeSJ}
\end{equation}
 where 
  \begin{equation*}
\bar{W}_t:=\tilde{W}_{t+t_1}-\tilde{W}_{t_1}.
\end{equation*}
Obviously, $\bar{W}=\{\bar{W}_t,{\cal F}_{t+t_1}\}_{0\le t\le T-t_1}$ is also a Brownian motion under $\mathbb Q$. Comparing (\ref{BS-SJ}) with (\ref{BS-tildeSJ}), we have
 \begin{equation}
\begin{aligned}
S_j=\tilde{S}_j\exp\left( \omega t_1+\sigma\tilde{W}_{t_1}  \right).
\end{aligned}
\end{equation}
Let $\bar{\bm W}=(\bar{W}_{t_2-t_1},\cdots,\bar{W}_{t_d-t_1})'$. Because $\tilde{W}_{t_1}$ and $\bar{\bm W}$ are independent and normally distributed, we can generate them by
 \begin{equation}
\begin{aligned}
\tilde{W}_{t_1}=\sqrt{t_1}X_1,\quad X_1\sim N(0,1),
\end{aligned}
\end{equation}
and
 \begin{equation}
\begin{aligned}
\bar{\bm W}={\bm AZ},\quad {\bm Z}\sim N({\bm 0}_{d-1},{\bm I}_{d-1}),
\end{aligned}
\end{equation}
where ${\bm Z}=(X_2,\cdots,X_d)'$, ${\bm 0}_{d-1}$ is a $(d-1)$-dimensional zero column vector, ${\bm I}_{d-1}$ is a $(d-1)\times (d-1)$ identify matrix, and ${\bm A}$ is a $(d-1)\times(d-1)$ matrix satisfying ${\bm AA}'={\bm \Sigma}$ with
 \begin{equation*}
\begin{aligned}
{\bm \Sigma}=
\left(
      \begin{array}{cccc}
        t_2-t_1 & t_2-t_1 & \cdots  & t_2-t_1 \\
        t_2-t_1 & t_3-t_1 & \cdots  & t_3-t_1 \\
        \vdots  & \vdots  & \ddots  & \vdots  \\
        t_2-t_1 & t_3-t_1 & \cdots  & t_d-t_1
      \end{array}
      \right)         .
\end{aligned}
\end{equation*}
Different choice of ${\bm A}$ may lead to different efficiency of QMC-based method. Since the Greek estimates can be viewed as a function  of ${\bm X}$, namely, ${\mathbb E}_{\mathbb Q}[\hat\theta_d]={\mathbb E}_{\mathbb Q}[g({\bm X})]$, and we have shown that the CMV estimate $G({\bm Z})={\mathbb E}_{\mathbb Q}[g({\bm X})|{\bm Z}]$ is continuous for common options, some PGMs (such as GPCA) can be used to reduce the effective dimension of $G({\bm Z})$.

Let ${\bm X}=(X_1,{\bm Z}')'$. Its components are independently and normally distributed. Let ${\bm S}=(S_1,\cdots,S_d)$. We have
\begin{equation}
\tilde{S}_j=S_0\exp\left(\omega(t_j-t_1)+\sigma {\bm A}_{j-1}{\bm Z}\right),
\end{equation}
and
\begin{equation}
\begin{aligned}                                               
{S}_j&=S_0\exp\left(\omega t_j+\sigma \sqrt{t_1}X_1+\sigma {\bm A}_{j-1}{\bm Z}\right)=\exp\left(\omega t_1+\sigma\sqrt{t_1}X_1\right)\tilde{S}_j,
\end{aligned}
\label{SJwithTildeSJ}
\end{equation}
where ${\bm A}_0={\bm 0}_{d-1}'$ is a zero row vector and ${\bm A}_j$ is the $j$th row of $\bm A$. The variable $S_j$ is a function of $\bm X$, with the component $ X_1$ and $\bm Z$ being separated ($\tilde S_j$ is a function of $\bm Z$).

By the above method, we achieve the goal of variables separation. In the expression of the Greek estimates ${\mathbb E}_{\mathbb Q}[g({\bm X})]$, the variable separation condition is usually satisfied for $g$. In some cases, we can calculate the analytical expression of $G({\bm Z})$.

 \begin{example}
\label{ex1}
As a warm-up example, we consider the binary Asian option whose payoff function is given by
\begin{equation*}
f(\bar{S}_T)={\bm 1}\{\bar{S}_T>K\}, 
\end{equation*}
where $K$ is the strike price. By Corolarry \ref{L1C} and (\ref{greek_f2}), the simulation estimate of $delta$ is given by
\begin{equation*}
\mathbb{E}_{\mathbb Q}[\hat \theta_d]=\mathbb{E}_{\mathbb Q}\left[\frac{2e^{-rT}}{S_0 \sigma^2}\left(   \frac{S_d-S_0}{T  S_A}-\omega \right)  {\bm 1}\{{S}_A>K\}  \right].
\end{equation*}

Since an indicator function is involved in the integrand, a direct use of QMC method could lead to large loss of efficiency. Based on (\ref{SJwithTildeSJ}), we have
\begin{equation*}
\begin{aligned}
S_A&=\exp\left(\omega t_1+\sigma\sqrt{t_1}X_1\right)\frac{1}{d}\sum_{j=1}^d\tilde{S}_j=\exp\left(\omega t_1+\sigma\sqrt{t_1}X_1\right)\tilde{S}_A,
\end{aligned}
\end{equation*}
where 
\begin{equation*}
\begin{aligned}
\tilde S_A:=\frac{1}{d}\sum_{j=1}^d\tilde{S}_j
=\frac{1}{d}\sum_{j=1}^dS_0\exp\left(\omega(t_j-t_1)+\sigma{\bm A}_{j-1}{\bm Z}\right).
\end{aligned}
\end{equation*}
Thus we have
\begin{equation*}
\begin{aligned}
\{S_A>K\}=\{X_1>\psi_d\},
\end{aligned}
\end{equation*}
where
\begin{equation*}
\begin{aligned}
\psi_d:=\frac{\ln K-\ln \tilde{S}_A-\omega t_1}{\sigma\sqrt{t_1}}
\end{aligned}
\end{equation*}
is a function of ${\bm Z}$. So we achieve variable separation, and the $delta$ can be written by
\begin{equation*}
\mathbb{E}_{\mathbb Q}[\hat \theta_d]= \mathbb{E}_{\mathbb Q}\left[\frac{2e^{-rT}}{S_0 \sigma^2}\left(   \frac{S_d-S_0}{T S_A}-\omega \right)  {\bm 1}\{X_1>\psi_d\}  \right]=:\mathbb{E}_{\mathbb Q}\left[g({\bm X})\right],
\end{equation*}
where $g({\bm x})$ has the form (\ref{4-2}).

We denote the density and distribution function of one-dimensional normal distribution by $\phi$ and $\Phi$, respectively. To use CMV method (\ref{4-5}), we need to calculate the analytical expression of $\mathbb{E}_{\mathbb{Q}}\left[  g({\bm X})|{\bm Z} \right]$. This can be done as follows
\begin{equation*}
\begin{aligned}
&\mathbb{E}_{\mathbb{Q}}\left[  g({\bm X})|{\bm Z}\right]\\
=&\int_{\psi_d}^{\infty} \frac{2e^{-rT}}{S_0 \sigma^2}\left(   \frac{S_d-S_0}{T S_A}-\omega \right)\phi(x_1)\mathrm{d}x_1\\
=&\frac{2e^{-rT}\tilde{S}_d}{TS_0\tilde S_A\sigma^2}\Phi(-\psi_d)-\frac{2e^{-rT-rt_1+\sigma^2t_1}}{T\tilde S_A\sigma^2}\Phi(-\sigma\sqrt{t_1}-\psi_d)-\frac{2e^{-rT}\omega}{S_0\sigma^2}\Phi(-\psi_d)\\
=&:G({\bm Z}).
\end{aligned}
\end{equation*}
Then $G({\bm Z})$ is the CMV estimate of $delta$ for the binary Asian option. Clearly, $G({\bm Z})$ has good smoothness.

Similarly, we can calculate CMV estimates of $gamma$ and $vega$ for the binary Asian option. To sum up, we give the following theorem.

\begin{theorem}\label{estimate1}
For the binary Asian option, the CMV estimates of $delta$, $gamma$ and $vega$ are given by
\begin{align*}
 &delta: G_{\Delta,1}(K;{\bm Z})=\left(\frac{2e^{-rT}\tilde{S}_d}{TS_0\tilde S_A\sigma^2} - \frac{2e^{-rT}\omega}{S_0\sigma^2} \right)\Phi_1-\frac{2e^{-rT-rt_1+\sigma^2t_1}}{T\tilde S_A\sigma^2}\Phi_2;\\
&gamma: G_{\Gamma,1}(K;{\bm Z})=\frac{4e^{-rT}}{\sigma^4S_0^2T^2} \Bigg[      \left(\frac{\tilde S_d^2}{\tilde S_A^2}+\omega rT^2-\frac{2rT\tilde S_d}{\tilde S_A}\right)\Phi_1 +\frac{S_0^2}{\tilde{S}_A^2}e^{-2rt_1+3\sigma^2t_1}\Phi_3   
 +\left( \frac{2\omega T S_0}{\tilde S_A}-  \frac{2S_0\tilde S_d}{\tilde S_A^2}\right)  e^{-rt_1+\sigma^2 t_1}\Phi_2      \Bigg] ;\\
 & vega: G_{\vartheta,1}(K;{\bm Z})=\Phi_1\Bigg( \frac{2e^{-rT}(\tilde S_d-(r-\sigma^2)T\tilde S_A)}{\sigma^2T\tilde S_A^2 d}  \sum_{j=1}^d\tilde{S}_j({\bm A}_{j-1}{\bm Z}-\sigma t_j)      -\frac{2e^{-rT}}{d^2\tilde S_A^2}\sum_{i=1}^d\sum_{j=i}^d\tilde S_i\tilde S_j({\bm A}_{j-1}{\bm Z}- \sigma t_j)  - \frac{e^{-rT}}{\sigma} \Bigg)  \\   
&\hspace{8em}  \  +\Phi_2 \Bigg( \frac{2e^{-rT}S_0\sqrt{t_1}}{\sigma^2T\tilde S_A}   \sigma\sqrt{t_1}e^{\sigma^2t_1/2-\omega t_1}-\frac{2S_0e^{-rT}}{\sigma^2Td\tilde{S}_A^2} e^{-rt_1+\sigma^2t_1} \sum_{j=1}^d\tilde{S}_j({\bm A}_{j-1}{\bm Z}-\sigma t_j)\Bigg)\\
&\hspace{8em} \ + \frac{1}{\sqrt{2\pi}}e^{-\psi_d^2/2}\Bigg(\frac{2e^{-rT}\tilde S_d\sqrt{t_1}}{\sigma^2 T\tilde S_A}
   -\frac{2e^{-rT}}{\sigma^2}\sqrt{t_1}(r-\sigma^2) -\frac{2e^{-rT}\sqrt{t_1}}{\tilde S_A^2 d^2}\sum_{i=1}^d\sum_{j=i}^d \tilde S_i \tilde S_j \Bigg) \\ 
&\hspace{8em} \  -\frac{2e^{-rT}S_0\sqrt{t_1}}{\sigma^2T\tilde S_A} \frac{1}{\sqrt{2\pi}}e^{\sigma^2t_1 /2-\omega t_1}  e^{-(\psi_d+\sigma\sqrt{t_1})^2/2} ,
\end{align*}
where $\Phi_1:=\Phi(-\psi_d)$, $\Phi_2:=\Phi(-\sigma\sqrt{t_1}-\psi_d)$ and $\Phi_3:=\Phi(-2\sigma\sqrt{t_1}-\psi_d)$.
\end{theorem}

We can easily see that all CMV estimates given above are infinitely times differentiable, which are QMC-friendly. On the other hand, based on Corollary \ref{L1C}, the original Malliavin estimates of the Greeks involve discontinuities. Thus we may expect that CMV method could provide better results than the original Malliavin method in QMC. 

As an example, for $delta$, the QMC-CMV estimate is given by
\begin{equation*}
\begin{aligned}
delta \approx   \frac{1}{n}\sum_{i=1}^n \Bigg[&   \frac{2e^{-rT}\tilde{S}_d(\Phi^{-1}({\bm u_i}))}  {TS_0\tilde S_A(\Phi^{-1}({\bm u_i}))\sigma^2}\Phi(-\psi_d(\Phi^{-1}({\bm u_i})))-\frac{2e^{-rT-rt_1+\sigma^2t_1}}{T\tilde S_A(\Phi^{-1}({\bm u_i}))\sigma^2}  \Phi(-\sigma\sqrt{t_1}-\psi_d(\Phi^{-1}({\bm u_i})))\\
& -\frac{2e^{-rT}\omega}{S_0\sigma^2}\Phi(-\psi_d(\Phi^{-1}({\bm u_i})))   \Bigg],
\end{aligned}
\end{equation*}
where ${\bm u}_1,\cdots,{\bm u}_n$ are QMC points. 
\end{example}

\begin{example}
\label{ex2}
Consider the Asian call option with the following payoff function
\begin{equation*}
f(\bar{S}_T)= (\bar{S}_T-K)^+.
\end{equation*}
Using the notations and imitating the steps in Example \ref{ex1}, we can obtain the CMV estimates for the Asian call option Greeks in the following theorem.

\begin{theorem}\label{estimate2}
For the Asian call option, the CMV estimates for $delta$, $gamma$ and $vega$ are given by
\begin{align*}
 &delta: G_{\Delta,2}(K;{\bm Z})=-\frac{2e^{-rT}}{T \sigma^2}\Phi_1+\left(\frac{2e^{-rT}\tilde{S}_d}{TS_0 \sigma^2}e^{rt_1}-  \frac{2e^{-rT}\omega \tilde{S}_A}{S_0\sigma^2}e^{rt_1} \right)    \tilde\Phi_2
 -KG_{\Delta,1}(K;{\bm Z});\\
& gamma: G_{\Gamma,2}(K;{\bm Z})=\left(\frac{8\omega e^{-rT}}{\sigma^4 S_0T}-\frac{8e^{-rT}\tilde S_d}{\sigma^4S_0T^2\tilde S_A}\right)\Phi_1+\frac{4e^{-rT}}{\sigma^4T^2\tilde S_A}e^{-rt_1+\sigma^2t_1}\Phi_2\\
 &\hspace{9em}\ \ +  \Bigg( \frac{4e^{-rT}\tilde S_d^2}{\sigma^4S_0^2T^2\tilde S_A}e^{rt_1}+ \frac{4e^{-rT}\omega r\tilde S_A}{\sigma^4 S_0^2}e^{rt_1}-\frac{8r\tilde S_de^{-rT}}{\sigma^4S_0^2T}e^{rt_1}\Bigg)\tilde\Phi_2 -KG_{\Gamma,1}(K;{\bm Z});\\
& vega: G_{\vartheta,2}(K;{\bm Z})=-\frac{2e^{-rT}S_0}{\sigma^2T\tilde S_A d}\sum_{j=1}^d\tilde S_j({\bm A}_{j-1}{\bm Z}-\sigma t_{j})\Phi_1-KG_{\vartheta,1}(K;{\bm Z}) -\frac{2e^{-rT}S_0\sqrt{t_1}}{\sigma^2T}\frac{1}{\sqrt{2\pi}}e^{-\psi_d^2/2}  \\
&\hspace{2em}+\Bigg[\frac{2e^{-rT}\tilde S_d}{\sigma^2T\tilde S_A d}\sum_{j=1}^d \tilde S_j({\bm A}_{j-1}{\bm Z}  -\sigma t_j)e^{rt_1}+\Bigg( \frac{2e^{-rT}\tilde S_d\sqrt{t_1}}{\sigma^2 T}-\frac{2e^{-rT}\tilde S_A}{\sigma^2}\sqrt{t_1}(r-\sigma^2)    \\
&\hspace{2em}-\frac{2e^{-rT}\sqrt{t_1}}{\tilde S_A d^2}\sum_{i=1}^d\sum_{j=i}^d\tilde S_i\tilde S_j \Bigg) \sigma \sqrt{t_1} e^{\sigma^2 t/2+\omega t} - \frac{2e^{-rT}(r-\sigma^2)}{\sigma^2d}\sum_{j=1}^d\tilde S_j e^{rt_1}({\bm A}_{j-1}{\bm Z} -\sigma t_j) \\
&\hspace{2em}   -  \frac{2e^{-rT}}{d^2\tilde S_A}\sum_{i=1}^d\sum_{j=i}^d \tilde S_i\tilde S_j ({\bm A}_{j-1}{\bm Z}-\sigma t_j)e^{rt_1}   - \frac{\tilde S_Ae^{-rT}}{\sigma}e^{rt_1}   \Bigg]\tilde\Phi_2 +\Bigg( \frac{2e^{-rT}\tilde S_d\sqrt{t_1}}{\sigma^2 T}-\frac{2e^{-rT}\tilde S_A}{\sigma^2}\sqrt{t_1}(r-\sigma^2) \\
&\hspace{2em}- \frac{2e^{-rT}\sqrt{t_1}}{\tilde S_A d^2}\sum_{i=1}^d\sum_{j=i}^d\tilde S_i\tilde S_j \Bigg)     \frac{1}{\sqrt{2\pi}}  e^{\sigma^2t/2+\omega t-(\psi_d-\sigma \sqrt{t_1})^2/2},
\end{align*}
where $\Phi_1:=\Phi(-\psi_d)$, $\Phi_2:=\Phi(-\sigma\sqrt{t_1}-\psi_d)$, and $\tilde\Phi_2:=\Phi(\sigma\sqrt{t_1}-\psi_d)$; and $G_{\Delta,1}$, $G_{\Gamma,1}$ and $G_{\vartheta,1}$ are the CMV estimates for the binary Asian option Greeks in Theorem \ref{estimate1}. 
\end{theorem}
 
 Note that in this case, the CMV estimates are the asymptotic unbiased estimates of the Greeks with respect to $d$ according to Theorem \ref{THM:bias}. We can easily see that all estimates given above are infinitely times differentiable, which are QMC-friendly as well.
\end{example}

\begin{example}
\label{ex3}
Consider the up-and-out Asian call option with the following payoff function
\begin{equation*}
f(\bar{S}_T)= (\bar{S}_T-K)^+{\bm 1}\{  \bar S_T\le H  \}.
\end{equation*}
Using the notations and imitating the steps in Example \ref{ex1}, we can obtain the CMV estimates for  the up-and-out Asian call option Greeks in the following theorem.

\begin{theorem}\label{estimate3}
For the up-and-out Asian call option, the CMV estimates for $delta$, $gamma$ and $vega$ are given by
\begin{align*}
& delta: G_{\Delta,3}(K,H;{\bm Z})=  G_{\Delta,2}(K;{\bm Z})- (G_{\Delta,2}(H;{\bm Z})+(H-K)G_{\Delta,1}(H;{\bm Z}));\\
&gamma: G_{\Gamma,3}(K,H;{\bm Z})=   G_{\Gamma,2}(K;{\bm Z})- (G_{\Gamma,2}(H;{\bm Z})+(H-K)G_{\Gamma,1}(H;{\bm Z}))  ;\\
& vega: G_{\vartheta,3}(K,H;{\bm Z})= G_{\vartheta,2}(K;{\bm Z})- (G_{\vartheta,2}(H;{\bm Z})+(H-K)G_{\vartheta,1}(H;{\bm Z})) ,
\end{align*}
where $G_{\Delta,1}$, $G_{\Gamma,1}$ and $G_{\vartheta,1}$ are the CMV estimates for the  binary Asian option Greeks in Theorem \ref{estimate1}, and $G_{\Delta,2}$, $G_{\Gamma,2}$ and $G_{\vartheta,2}$ are the CMV estimates for the Greeks in Theorem \ref{estimate2}. 
\end{theorem}

 Again, we can easily see that all estimates given above are infinitely times differentiable. 

\end{example}

\section{Numerical experiments}
\label{sec:NE}
n this section, we present numerical experiments for Examples \ref{ex1}-\ref{ex3} in Section \ref{Sub_ill} by comparing the variance reduction factors (VRFs) of different methods to illustrate the effectiveness of  QMC-CMV method . We choose RQMC points instead of QMC points since the RQMC points are random which allow the variance to be estimated.   
 
Let ${\cal P}=\{{\bm u}_1,\cdots,{\bm u}_N\}$ be the $(d-1)$-dimensional QMC points.We generate $M$ independent random versions ${\cal P}_j=\{{\bm u}_1^{(j)},\cdots,{\bm u}_N^{(j)}\}$, $j=1,\cdots, M$. For each batch ${\cal P}_j$, we compute the estimate
 \begin{equation*}
\begin{aligned}
Q_{N}^{(j)}=\frac{1}{N}\sum_{i=1}^N G({\bm u}_i^{(j)}),
\end{aligned}
\end{equation*}
where $G({\bm Z})$ is the payoff function. The ultimate estimate is given by
 \begin{equation*}
\begin{aligned}
\tilde Q=\frac{1}{M}\sum_{j=1}^M Q_N^{(j)},
\end{aligned}
\end{equation*}
and the variance of the estimate $\tilde Q$ is estimated by
 \begin{equation*}
\begin{aligned}
\sigma^2_{\text{RQMC}}=\frac{1}{M(M-1)}\sum_{j=1}^M(Q_{N}^{(j)}-\tilde Q)^2.
\end{aligned}
\end{equation*}
In our experiments, the number of samples $N$ and the number of batches $M$ are the same for each method. We take $H=120$ in the up-and-out Asian call options. In each numerical experiment, there are $M=500$ independent batches of simulation and each batch contains $N=2^{15}$ samples. The Malliavin (MV) estimate in MC is chosen as the benchmark, and we denote it by MC-MV. Let the variance of the MC-MV be $V_0^2$. If the variance of an alternative estimate is $V^2$, then the VRF of this estimate with respect to MC-MV is given by
 \begin{equation*}
\begin{aligned}
\text{VRF}:=\frac{V_0^2}{V^2}.
\end{aligned}
\end{equation*}

In QMC setting, GPCA method is used as the PGM to reduce the effective dimension (due to its superiority in Section \ref{Sec-QMC}) and the scrambled Sobol' points proposed by \citet{Matousek98} are used as the RQMC points.  

As pointed out by \citet{Hull22}, a risky asset typically has a volatility between $15\%$ and $60\%$. Following the usual parameters selection in simulation and without loss of generality (see \cite{Glasserman04,Fournie99,Zhang20,Xiao18}), we set $S_0=100$, $\sigma=0.2, 0.4$, $r=0.1$, $T=1$ for all examples and set $d=64, 256$, and $K=90, 100, 110$. The option is in the money if $K=90$, at the money if $K=100$, and out of money if $K=110$. 
 
For each Greek, we compare the MV method and the CMV method. The MV method includes MC-MV and QMC-MV, which combine MV method with MC and QMC, respectively. The CMV method includes MC-CMV and QMC-CMV, which combine CMV method with MC and QMC, respectively.

In this paper, all the programs are run under macOS using MATLAB (version 8.4.0). The computer is a MacBook Pro laptop with Intel Core i5 Duo CPU @ 2.4 GHz, 8 GB of RAM. 

The results for binary Asian options, Asian call options and up-and-out Asian call options are presented in Tables \ref{table1}-\ref{table3-2}, respectively . We observe the following from the experiments.

\begin{table} 
\caption{The VRFs for binary Asian options with payoff function $f={\bm 1}\{\bar S_T>K\}$.}\label{table1}
\begin{tabular}{lllllll}
\hline\noalign{\smallskip}
 Greeks & K & d &MC-MV &QMC-MV &MC-CMV &QMC-CMV\\
\noalign{\smallskip}\hline\noalign{\smallskip}
delta & 90 & 64&1            &11&1&\bf 55,545 \\ 
  &  & 128   &1  &12&1&\bf 41,135 \\ 
  & 100    & 64     &1  &11 &1&\bf 35,126 \\ 
   &  & 128  &1  &8&1&\bf 21,778 \\ 
    & 110   & 64    &1  & 8&1&\bf 38,228 \\ 
     &    &128    &1  &7&1&\bf 28,820 \\ 
     \noalign{\smallskip}\hline\noalign{\smallskip}
     gamma & 90 & 64&1          &5&1&\bf 5,647 \\ 
  &  & 128   &1  &3&1&\bf 4,113 \\ 
  & 100    & 64     &1  &6 &1&\bf 7,154 \\ 
   &  & 128  &1  &3&1&\bf 7,508 \\ 
    & 110   & 64    &1  &4&1&\bf  4,678\\ 
     &    &128    &1  &3&1&\bf 3,565 \\ 
   \noalign{\smallskip}\hline\noalign{\smallskip}
  vega& 90 & 64&1            &5&1&\bf 8,139 \\ 
  &  & 128   &1  &4&1&\bf 5,742 \\ 
  & 100    & 64     &1  &4 &1&\bf 6,582 \\ 
   &  & 128  &1  &4&1&\bf 6,738 \\ 
    & 110   & 64    &1  & 4&1&\bf 6,154 \\ 
     &    &128    &1  &3&1&\bf 6,415 \\ 
\noalign{\smallskip}\hline
\end{tabular}
\end{table}

\begin{table} 
\caption{The VRFs for binary Asian options with payoff function $f={\bm 1}\{\bar S_T>K\}$ with $\sigma=0.4$.}\label{table1_sigma}
\begin{tabular}{lllllll}
\hline\noalign{\smallskip}
 Greeks & K & d &MC-MV &QMC-MV &MC-CMV &QMC-CMV\\
\noalign{\smallskip}\hline\noalign{\smallskip}
delta & 90 & 64&1            &11 &1&\bf 32,601 \\ 
  &  & 128   &1  &9 &1&\bf 18,763 \\ 
  & 100    & 64     &1  &12 &1&\bf 29,189 \\ 
   &  & 128  &1  &7&1&\bf 14,577 \\ 
    & 110   & 64    &1  & 8&1&\bf 29,624 \\ 
     &    &128    &1  &6&1&\bf 17,452 \\ 
     \noalign{\smallskip}\hline\noalign{\smallskip}
     gamma & 90 & 64&1          &5&1&\bf 4,368 \\ 
  &  & 128   &1  &3&1&\bf 3,966 \\ 
  & 100    & 64     &1  &4 &1&\bf 4,517 \\ 
   &  & 128  &1  &3&1&\bf 3,642 \\ 
    & 110   & 64    &1  &4&1&\bf  3,013\\ 
     &    &128    &1  &3&1&\bf 2,726 \\ 
   \noalign{\smallskip}\hline\noalign{\smallskip}
  vega& 90 & 64&1            &4&1&\bf 5,675 \\ 
  &  & 128   &1  &3&1&\bf 6,019 \\ 
  & 100    & 64     &1  &3 &1&\bf 5,209 \\ 
   &  & 128  &1  &3&1&\bf 5,027 \\ 
    & 110   & 64    &1  & 4&1&\bf 3,907 \\ 
     &    &128    &1  &3&1&\bf 3,937 \\ 
\noalign{\smallskip}\hline
\end{tabular}
\end{table}

 \begin{table} 
\caption{The Greek values  ($\times 10^{-3}$) for binary Asian options with payoff function $f={\bm 1}\{\bar S_T>K\}$.}\label{table1-1}
\begin{tabular}{lllllll}
\hline\noalign{\smallskip}
 Greeks & K & d &MC-MV &QMC-MV &MC-CMV &QMC-CMV\\
\noalign{\smallskip}\hline\noalign{\smallskip}
delta & 90 & 64&  14.083    &14.053&14.071&14.056 \\ 
  &  & 128   &   14.024      &14.016&13.999&14.015  \\ 
  & 100    & 64     &    29.204      &29.188&29.187&29.185   \\ 
   &  & 128  &     29.179       &  29.199 &   29.193     & 29.198\\ 
    & 110   & 64   &   27.802 & 27.802&27.812&27.805\\ 
     &    &128    &27.747     &27.728&27.747&27.730\\ 
     \noalign{\smallskip}\hline\noalign{\smallskip}
gamma & 90 & 64&-1.7400          &-1.7407&-1.7414& -1.7404 \\ 
  &  & 128   &-1.7216  &-1.7195&-1.7184&-1.7185 \\ 
  & 100    & 64     &-1.1893  &-1.1895 &-1.1890&-1.1896 \\ 
   &  & 128    &-1.1690&-1.1693&-1.1688 &-1.1713 \\ 
    & 110   & 64    &0.8000  &0.8018&0.8035&0.8010\\ 
     &    &128    &0.8229 &0.8208&0.8203&0.8196 \\ 
          \noalign{\smallskip}\hline\noalign{\smallskip}
vega& 90 & 64&-1092.7            &-1095.0&-1096.1&-1095.0\\ 
  &  & 128   &-1107.2  &-1109.5 &-1111.2 &-1110.3  \\ 
  & 100    & 64     & -831.14&-831.92  &-833.07  &-831.66 \\ 
   &  & 128   &-839.68&-840.38&  -839.58&-840.47\\ 
    & 110   & 64    &512.06  & 511.82&512.05&512.27\\ 
     &    &128    &510.34  &509.42&511.49&510.53 \\ 
\noalign{\smallskip}\hline
\end{tabular}
\end{table}

  \begin{table} 
\caption{The simulation times (unit: s) for binary Asian options with payoff function $f={\bm 1}\{\bar S_T>K\}$.}\label{table1-2}
\begin{tabular}{lllllll}
\hline\noalign{\smallskip}
 Greeks & K & d &MC-MV &QMC-MV &MC-CMV &QMC-CMV\\
\noalign{\smallskip}\hline\noalign{\smallskip}
delta & 90 & 64&   157         &172 & 123& 152 \\ 
  &  & 128   &  301   & 356&241 &302   \\ 
  & 100    & 64     &     147      &179 & 117&  150  \\ 
   &  & 128  &   309        & 369& 247      &312 \\ 
    & 110   & 64    &   156& 187 & 124&153 \\ 
     &    &128   &305&  370 & 252& 329  \\ 
   \noalign{\smallskip}\hline\noalign{\smallskip}
  gamma & 90 & 64&153           &191&132&156 \\ 
  &  & 128   &304         &387&267&348 \\ 
  & 100     &64&176          &202&132&173 \\ 
   &  & 128  &337&415&287&360 \\ 
    & 110   & 64    &176  &200&139&175\\ 
     &    &128    &344  &410&290&354\\ 
   \noalign{\smallskip}\hline\noalign{\smallskip}
  vega& 90 & 64&968            &949&1601&1794\\ 
  &  & 128   &3170  &3884&7197&7091 \\ 
  & 100    & 64     &930 &919 &1517&1549\\ 
   &  & 128  &3074  &3625&7190&6990 \\ 
    & 110   & 64    &994 & 1022&1615  &1644 \\ 
     &    &128    &3319  &3426&6897&6951 \\ 
\noalign{\smallskip}\hline
\end{tabular}
\end{table}

 \begin{table} 
\caption{The VRFs for Asian call options with payoff function $f=(\bar S_T-K)^+$.}\label{table2}
\begin{tabular}{lllllll}
\hline\noalign{\smallskip}
 Greeks & K & d &MC-MV &QMC-MV &MC-CMV &QMC-CMV\\
\noalign{\smallskip}\hline\noalign{\smallskip}
delta & 90 & 64            &1  &50&1&\bf 12,804 \\ 
  &  & 128   &1  &22 &1  &\bf 13,396 \\ 
  & 100    & 64     &1  &17 &1&\bf 5,809 \\ 
   &  & 128  &1  &12&1&\bf 6,695 \\ 
    & 110   & 64    &1  &9 &1&\bf 5,007\\ 
     &    &128    &1  &7&1&\bf 4,987 \\ 
   \noalign{\smallskip}\hline\noalign{\smallskip}
  gamma & 90 & 64&1          &5&1&\bf 919 \\ 
  &  & 128   &1  &3&1&\bf 1,132 \\ 
  & 100    & 64     &1  &3 &1&\bf 727 \\ 
   &  & 128  &1  &3&1&\bf 739 \\ 
    & 110   & 64    &1  & 3&1&\bf 481 \\ 
     &    &128    &1  &2&1&\bf 494 \\ 
   \noalign{\smallskip}\hline\noalign{\smallskip}
  vega& 90 & 64&1            &4&1&\bf 754 \\ 
  &  & 128   &1  &4&1&\bf 724 \\ 
  & 100    & 64     &1  &3 &1&\bf 537\\ 
   &  & 128  &1  &3&1&\bf 643 \\ 
    & 110   & 64    &1  & 3&1&\bf 437 \\ 
     &    &128    &1  &2&1&\bf 349\\ 
     \noalign{\smallskip}\hline
\end{tabular}
\end{table}

 \begin{table} 
\caption{The Greek values  ($\times 10^{-3}$) for Asian call options with payoff function $f=(\bar S_T-K)^+$.}\label{table2-1}
\begin{tabular}{lllllll}
\hline\noalign{\smallskip}
 Greeks & K & d &MC-MV &QMC-MV &MC-CMV &QMC-CMV\\
\noalign{\smallskip}\hline\noalign{\smallskip}
delta & 90 & 64& 884.51           &883.76&884.44&883.83 \\ 
  &  & 128   & 879.06       &879.46&879.72&879.28  \\ 
  & 100    & 64     &659.73            &660.31&659.90&660.19   \\ 
   &  & 128  &   655.69            &  655.75 &   655.53     & 655.78 \\ 
    & 110   & 64    &359.25  & 359.40&359.35&359.42\\ 
     &    &128    &355.36 &355.11&355.13&355.18 \\ 
   \noalign{\smallskip}\hline\noalign{\smallskip}
  gamma & 90 & 64&12.113          &12.196&12.101&12.157 \\ 
  &  & 128   &12.500  &12.351&12.349&12.365 \\ 
  & 100    & 64     &29.165 &29.220 &29.258&29.160 \\ 
   &  & 128  &29.218  &29.117&29.130&29.187 \\ 
    & 110   & 64    &30.802  &30.821&30.906&30.773\\ 
     &    &128    &30.589  &30.590&30.579&30.596\\ 
   \noalign{\smallskip}\hline\noalign{\smallskip}
  vega& 90 & 64&9119.8           &9090.1&9005.7&9069.4 \\ 
  &  & 128   &8552.1  &8530.9&8638.5&8563.6\\ 
  & 100    & 64     &20364  &20397 &20346&20379 \\ 
   &  & 128  &19939  &19921&19874&19922\\ 
    & 110   & 64    &21823  & 21773&21820&21796 \\ 
     &    &128    &21497  &21480&21470&21458 \\ 
\noalign{\smallskip}\hline
\end{tabular}
\end{table}

  \begin{table} 
\caption{The simulation times (unit: s) for Asian call options with payoff function $f=(\bar S_T-K)^+$.}\label{table2-2}
\begin{tabular}{lllllll}
\hline\noalign{\smallskip}
 Greeks & K & d &MC-MV &QMC-MV &MC-CMV &QMC-CMV\\
\noalign{\smallskip}\hline\noalign{\smallskip}
delta & 90 & 64& 143         &168&119&150\\ 
  &  & 128   & 286          &376&243&316  \\ 
  & 100    & 64     &152           &173&119&150  \\ 
   &  & 128  &     291          &  379  &  245    & 324 \\ 
    & 110   & 64    &  152&  182  &  125  & 164  \\ 
     &    &128    &301 &386&259&347\\ 
   \noalign{\smallskip}\hline\noalign{\smallskip}
  gamma & 90 & 64&159         &198&142&177 \\ 
  &  & 128   &313 &392&266&334 \\ 
  & 100    & 64     &153  &187 &137&172\\ 
   &  & 128  &299  &375&254&321 \\ 
    & 110   & 64    &147  &183&131&163\\ 
     &    &128    &292  &354&249&326\\ 
   \noalign{\smallskip}\hline\noalign{\smallskip}
  vega& 90 & 64&990            &1101&1772&1762\\ 
  &  & 128   &3412 &3415&6808&6862 \\ 
  & 100    & 64     &987  &998 &1704&1734 \\ 
   &  & 128  &3315  &3384&6813&7045 \\ 
    & 110   & 64    &1106  & 1182&2081&1739 \\ 
     &    &128    &3340  &3358&6185&6827 \\ 
\noalign{\smallskip}\hline
\end{tabular}
\end{table}

 \begin{table} 
\caption{The VRFs for up-and-out Asian call options with payoff function $f=(\bar S_T-K)^+{\bm 1}\{\bar S_T\le H\}$, where $H=120$.}\label{table3}
\begin{tabular}{lllllll}
\hline\noalign{\smallskip}
 Greeks & K & d &MC-MV &QMC-MV &MC-CMV &QMC-CMV\\
\noalign{\smallskip}\hline\noalign{\smallskip}
delta & 90 & 64&1            &2&1&\bf 28,707 \\ 
  &  & 128   &1  &2&1&\bf 20,327 \\ 
  & 100    & 64     &1  &2 &2& \bf  51,475 \\ 
   &  & 128  &1  &2&1&\bf 38,146\\ 
    & 110   & 64    &1  & 1&2&\bf 67,723 \\ 
     &    &128    &1  &1&1&\bf 37,032 \\ 
   \noalign{\smallskip}\hline\noalign{\smallskip}
  gamma & 90 & 64       &1   &2&1&\bf 1,355 \\ 
  &  & 128   &1  &1&1&\bf 1,052 \\ 
  & 100    & 64     &1  &2 &1&\bf 1,838 \\ 
   &  & 128  &1  &2&1&\bf 2,235 \\ 
    & 110   & 64    &1  & 1&2&\bf 3,032 \\ 
     &    &128    &1  &1&2&\bf 2,369 \\ 
   \noalign{\smallskip}\hline\noalign{\smallskip}
  vega& 90 & 64&1            &1&1&\bf 18,247 \\ 
  &  & 128   &1  &2&1&\bf 11,787 \\ 
  & 100    & 64     &1  &1 &2&\bf 18,495\\ 
   &  & 128  &1  &1&1&\bf 11,064 \\ 
    & 110   & 64    &1  &1  &2&\bf 12,718 \\ 
     &    &128    &1  &1&1&\bf 8,812 \\ 
     \noalign{\smallskip}\hline
\end{tabular}
\end{table}

  \begin{table} 
\caption{The Greek values  ($\times 10^{-3}$) for up-and-out Asian call options with payoff function $f=(\bar S_T-K)^+{\bm 1}\{\bar S_T\le H\}$, where $H=120$.}\label{table3-1}
\begin{tabular}{lllllll}
\hline\noalign{\smallskip}
 Greeks & K & d &MC-MV &QMC-MV &MC-CMV &QMC-CMV\\
\noalign{\smallskip}\hline\noalign{\smallskip}
delta & 90 & 64&      283.52      &283.79&283.73&283.64\\ 
  &  & 128   &    287.00      &286.92 &287.00&287.03 \\ 
  & 100    & 64     &  212.65            &212.83&212.77& 212.73   \\ 
   &  & 128  &   214.62     &   214.60& 214.70     & 214.58\\ 
    & 110   & 64    &64.706  &64.644&64.684&64.689\\ 
     &    &128    &64.987 &65.069&65.005&65.043  \\ 
   \noalign{\smallskip}\hline\noalign{\smallskip}
  gamma & 90 & 64&-47.024          &-47.043&-47.057&-47.035 \\ 
  &  & 128   &-46.568  &-46.552&-46.511&-46.533 \\ 
  & 100    & 64     &-16.453  &-16.452 &-16.468&-16.454 \\ 
   &  & 128  &-16.129  &-16.147&-16.165&-16.144 \\ 
    & 110   & 64    &-1.2664  &-1.2597&-1.2706&-1.2707\\ 
     &    &128    &-1.1690  &-1.1668&-1.1666&-1.1675 \\ 
   \noalign{\smallskip}\hline\noalign{\smallskip}
  vega& 90 & 64&-32905           &-32884&-32910&-32895 \\ 
  &  & 128   &-32828 &-32857&-32844&-32850 \\ 
  & 100    & 64     &-12151  &-12151 &-12154&-12156 \\ 
   &  & 128  &-12068  &-12062&-12080&-12071\\ 
    & 110   & 64    &-1152.0  & -1151.6&\-1152.5&-1151.1 \\ 
     &    &128    &-1111.0  &-1111.1&-1112.2&-1110.5\\ 
\noalign{\smallskip}\hline
\end{tabular}
\end{table}

  \begin{table} 
\caption{The simulation times (unit: s) for up-and-out Asian call options with payoff function $f=(\bar S_T-K)^+{\bm 1}\{\bar S_T\le H\}$, where $H=120$.}\label{table3-2}
\begin{tabular}{lllllll}
\hline\noalign{\smallskip}
 Greeks & K & d &MC-MV &QMC-MV &MC-CMV &QMC-CMV\\
\noalign{\smallskip}\hline\noalign{\smallskip}
delta & 90 & 64&    144       &167&128&153 \\ 
  &  & 128   &291         &350 &248&335  \\ 
  & 100    & 64     &150          &179&125&152  \\ 
   &  & 128  &   285            & 383 &   249  & 327\\ 
    & 110   & 64    &151  & 210&146&180\\ 
     &    &128    &333  &401&251&315 \\ 
   \noalign{\smallskip}\hline\noalign{\smallskip}
  gamma & 90 & 64&148          &202&170&173 \\ 
  &  & 128   &309  &340&261&326 \\ 
  & 100    & 64     &151  &197 &143&186 \\ 
   &  & 128  &289  &356&288&338 \\ 
    & 110   & 64    &150  &179&141&171\\ 
     &    &128    &303  &358&264&335\\ 
   \noalign{\smallskip}\hline\noalign{\smallskip}
  vega& 90 & 64&1034            &1053&1814&1926 \\ 
  &  & 128   &3364  &3541&7604&7482 \\ 
  & 100    & 64     &981  &995 &1745&1777 \\ 
   &  & 128  &3404  &3458&6985&7047 \\ 
    & 110   & 64    &1017  & 1057&1760&1788 \\ 
     &    &128    &3506  &3522&7635&7352 \\ 
\noalign{\smallskip}\hline
\end{tabular}
\end{table}

\begin{itemize}
 \item Among all methods, QMC-CMV method is the most effective method to calculate Greeks. This is because QMC-CMV method  improves the smoothness of the integrand, which is very significant in enhancing the computational efficiency of the QMC method. Furthermore, due to the improved smoothness of the integrands, GPCA method can be applied to reduce the effective dimension and increase the efficiency of the QMC method. 
\end{itemize}

\begin{itemize}
 \item QMC-CMV behaves much better than QMC-MV. However, in MC setting, the VRFs of  MC-CMV are nearly the same as  the VRFs of MC-MV. This is because MC method is not sensitive to the smoothness of the integrand. Moreover, CMV and MV in QMC setting behave better than those in MC setting, respectively. However, the VRFs of MV in QMC are no more than one hundred, while the VRFs of QMC-CMV can be larger than tens of thousands in some cases. This is because the integrands in CMV have good smoothness. 
\end{itemize}

\begin{itemize}
 \item The VRFs of QMC-CMV in binary Asian option and up-and-out Asian call option (Tables \ref{table1} and \ref{table3}) are more impressive than those in Asian call option (Table \ref{table2}). The reason is that the payoff function of Asian call option is continuous, whereas the payoff functions of binary Asian option and up-and-out Asian call option are not. Thus the effect of smoothing for binary Asian option and up-and-out Asian option is more significant than that for the Asian call option.
\end{itemize}

\begin{itemize}
 \item For $delta$ of all options (Tables \ref{table1}, \ref{table2} and \ref{table3}), the superiority of the QMC-CMV method is more impressive. The VRFs of QMC-CMV are larger than tens of thousands in most cases. For $gamma$ of all options (Tables \ref{table1}, \ref{table2} and \ref{table3}) and $vega$ of binary Asian options and Asian call options (Tables \ref{table1} and \ref{table2}), the VRFs of QMC-CMV are no more than ten thousands. However, for $vega$ of up-and-out Asian call options (Table \ref{table3}), the VRFs of QMC-CMV are larger than tens of thousands in most cases. A critical factor that leads to the above results is the different weight functions in different Greeks formulae (\ref{MV_formu}).
\end{itemize}

\begin{itemize}
 \item Dimensionality and strike price have some impact on the VRFs. For binary Asian options (Table \ref{table1}), for example, the VRFs of QMC-CMV for $delta$ are smaller when the dimension increases. For Asian call options (Table \ref{table2}), for example, the VRFs of QMC-CMV are smaller when the strike price is larger for all three Greeks. Generally, the high dimensionality has little impact on all methods in QMC setting since the GPCA method reduces the effective dimension, which lifts the curse of dimensionality.
\end{itemize}

\begin{itemize}
 \item From Tables \ref{table1-1}, \ref{table2-1} and \ref{table3-1}, we can see that there is no significant difference in Greek values between those methods. This implies the consistency of our estimates.
\end{itemize}

\begin{itemize}
 \item From Tables \ref{table1-2}, \ref{table2-2} and \ref{table3-2}, we can see that there is no significant difference in simulation times between those methods when considering $delta$ and $gamma$. However,  for $vega$, the simulation time of CMV is nearly twice as long as that of MV. A reasonable explanation is that the weight functions of $vega$ of the CMV estimates are much more complicated (see Theorem \ref{estimate1}, \ref{estimate2} and \ref{estimate3}).
\end{itemize}

\begin{itemize}
 \item From Tables \ref{table1} and \ref{table1_sigma}, we can see that the volatility has some impact on the VRFs. The VRFs of QMC-CMV are slightly smaller when the volatility increases. However, the robustness of our method with respect to the volatility still holds.
\end{itemize}

\begin{remark}
We only consider the one-dimensional simple Asian options as illustrations in Sections \ref{Sub_ill} and \ref{sec:NE} for convenience. However, the QMC-CMV method can also be implemented in the case of complex Asian options and multiple assets. Due to the complicated calculation, the results of those case are not presented in this paper. Moreover, since the weight functions of simple Asian option Greeks include $S_T$ as well (see Corollary \ref{L1C}) and there is no essential difference between the weight functions of simple and complex Asian options, the efficiency of the QMC-CMV in Section \ref{sec:NE} still holds when we deal with the complex Asian options.
\end{remark}


\section{Conclusion}
\label{sec:conclusion}

In this paper, we develop an integration by parts formula in the multi-dimensional Malliavin calculus and propose a new method to calculate the Greeks for Asian options. We obtain the Greeks formulae for both simple and complex continuous-time Asian options in the multi-asset Black-Scholes model by the integration by parts formula. By discretization we derive the simulation estimates of the Asian option Greeks. We prove the asymptotic convergence of the simulation estimates when the payoff function is nonnegative, continuous with linear growth. After then, by taking conditional expectations, we get the CMV estimates for Greeks. We show how to calculate the conditional expectations analytically for binary Asian options, Asian call options and up-and-out Asian call options. We find that the smoothness of the CMV estimate is enhanced, leading to a notable efficiency improvement in QMC. Even for options with continuous payoffs, taking conditional expectations is of great significance because the CMV estimates are smoothed. The numerical results show that, for binary Asian options, Asian call options and up-and-out Asian call options, the VRFs of QMC-CMV methods can be larger than tens of thousands in some cases compared with the MC-MV method.

For further work, the bias of the simulation estimates for Asian option Greeks in general cases need to be studied. More general formulae for CMV estimates of Asian option Greeks are also expected to be developed. Moreover, extending the QMC-CMV method for Greeks to other models is a subject of ongoing research.

 \section*{acknowledgements}
The authors are very grateful to the editors and the anonymous referee for their helpful suggestions and comments. The work is funded by the National Natural Science Foundation of China (No. 72071119).

 \section*{Conflict of interest}
 The authors declare that they have no conflict of interest. All authors certify that they have no affiliations with or involvement in any organization or entity with any financial interest or non-financial interest in the subject matter or materials discussed in this manuscript.

   \bibliographystyle{plain}

\end{CJK}
\end{document}